\documentclass[11pt,reqno]{amsart}
\usepackage{amsmath,amsthm}

\newtheorem{theorem}{Theorem}[section]
\newtheorem{lemma}[theorem]{Lemma}
\newtheorem{proposition}[theorem]{Proposition}
\newtheorem{corollary}[theorem]{Corollary}

\theoremstyle{definition}

\newtheorem{definition}[theorem]{Definition}

\newtheorem{remark}[theorem]{Remark}

\newcommand{\mc}{\mathcal}
\newcommand{\wh}{\widehat}

\numberwithin{equation}{section}

\allowdisplaybreaks[1]

\begin{document}

\title{Higher Affine Connections}

\author{David N. Pham}
\address{School of Mathematical Sciences, Rochester Institute of Technology, Rochester, NY 14623}
\curraddr{}
\email{dpham90@gmail.com}
\thanks{}



\date{}

\dedicatory{}

\begin{abstract}
For a smooth manifold $M$, it was shown in \cite{BPH} that every affine connection on the tangent bundle $TM$ naturally gives rise to covariant differentiation of multivector fields (MVFs) and differential forms along MVFs.   In this paper, we generalize the covariant derivative of \cite{BPH} and construct covariant derivatives along MVFs which are not induced by affine connections on $TM$.  We call this more general class of covariant derivatives \textit{higher affine connections}.   In addition, we also propose a framework which gives rise to non-induced higher connections; this framework is obtained by equipping the full exterior bundle $\wedge^\bullet TM$ with an associative bilinear form $\eta$.  Since the latter can be shown to be equivalent to a set of differential forms of various degrees, this framework also provides a link between higher connections and multisymplectic geometry.  
\end{abstract}

\maketitle
\section{Introduction}
Let $M$ be a manfiold.  It was shown in \cite{BPH} that every affine connection $\nabla$ on the tangent bundle $TM$ naturally gives rise to covariant differentiation of multivector fields (MVFs) and differential forms along MVFs.  For covariant differentiation of MVFs along MVFs, the covariant derivative of \cite{BPH} (which we will again denote as $\nabla$) satisfies
\begin{align}
\label{BPHPropertyA}
\nabla_{X\wedge Y} Z&=(-1)^kX \wedge \nabla_Y Z+(-1)^{(k-1)l} Y\wedge \nabla_X Z\\
\label{BPHPropertyB}
\nabla_X(Y\wedge Z)&=(\nabla_XY)\wedge Z+(-1)^{(k-1)l}Y\wedge\nabla_X Z
\end{align}
for $X\in \Gamma(\wedge^k TM)$, $Y\in \Gamma(\wedge^l TM)$, and $Z\in \Gamma(\wedge^\bullet TM)$.  Any covariant derivative of MVFs along MVFs which satisfies (\ref{BPHPropertyA}) and (\ref{BPHPropertyB}) is necessarily induced by an affine connection on the tangent bundle.  In this paper, we consider a class of covariant derivatives of MVFs along MVFs which satisfies all the main properties of those of \cite{BPH} except possibly (\ref{BPHPropertyA}) and (\ref{BPHPropertyB}).  We call such a covariant derivative a \textit{higher affine connection}.  Hence, the covariant derivative of MVFs along MVFs from \cite{BPH} can be seen as a special case of a higher affine connection. 

For covariant differentiation of differential forms along MVFs, the covariant derivative of \cite{BPH} has several nice properties which are summarized in Theorem 4.2 of \cite{BPH}.  To describe the approach of \cite{BPH}, we start with a decomposable $k$-vector field $X=X_1\wedge \cdots \wedge X_k$ and an affine connection $\nabla$ on $TM$.  For a differential form $\omega\in \Omega^\bullet(M)$, we define
\begin{equation}
\label{BPHDiffForm1}
\nabla_X \omega=\sum_{j=1}^k (-1)^{j+1}i_{X[j]} (\nabla_{X_j}\omega),
\end{equation}
where $X[j]:=X_1\wedge \cdots \wedge \widehat{X}_j\wedge \cdots \wedge X_k$ (the ``hat" denoting omission) and $i_{X[j]}$ is the interior product by $X[j]$.  For $f\in C^\infty(M)$ and $\sigma\in \mathfrak{S}_k$, a direct calculation shows that
\begin{equation}
\label{BPHDiffForm2}
\nabla_{fX^{\sigma}}\omega=\mbox{sgn}(\sigma)f\nabla_X\omega,
\end{equation}
where $X^\sigma:=X_{\sigma(1)}\wedge \cdots \wedge X_{\sigma(k)}$.  Since every $k$-vector field is locally a finite sum of decomposable $k$-vector fields, (\ref{BPHDiffForm2}) implies that the above definition extends to all $X\in \Gamma(\wedge^k TM)$ by $C^\infty(M)$-linearity.  Note that if $\omega$ is an $l$-form and $X$ a $k$-vector field with $l\ge k-1$, then $\nabla_X \omega$ is an $(l-k+1)$-form.   

Since a higher connection is not completely defined by an affine connection on $TM$, (\ref{BPHDiffForm1}) is not directly applicable to higher connections.  Even so, we show that higher connections do indeed allow for covariant differentiation of differential forms along MVFs.  Moreover, the result has properties similar to those of \cite{BPH}.

The motivation for discarding condition (\ref{BPHPropertyA}) stems from two sources: \textit{generalized geometry} \cite{H1} \cite{H2} and the basic idea at the root of \textit{string theory} \cite{GSW} \cite{Z}.  In generalized geometry, one replaces the tangent bundle with $TM \oplus T^\ast M$ and the Lie bracket with the Courant bracket.  In this setting, one looks for geometric structures on $TM \oplus T^\ast M$ which are analogues of the familiar objects one encounters in differential geometry.   In string theory, the notion of a point particle is replaced by one dimensional extended objects called \textit{strings}.  As a consequence of this, a worldline (i.e., the path that a point particle makes  through spacetime) is generalized to a 2-dimensional worldsheet (i.e., the surface that a string sweeps out as it propagates).  Let $X: \Sigma\rightarrow M$ be a worldsheet map, where $\Sigma \subset \mathbb{R}^2$ has coordinates $(\tau,\sigma)$, and let $p$ be a point on $X$.  One can think of $X$ as a higher dimensional worldline with the following ``tangent vectors" at $p$:
\begin{equation}
\nonumber
\partial_\tau X|_{(\tau_0,\sigma_0)}, \hspace*{0.1in} \partial_\sigma X|_{(\tau_0,\sigma_0)},\hspace*{0.1in} (\partial_\tau X \wedge \partial_\sigma X)_{(\tau_0,\sigma_0)},
\end{equation} 
where $X(\tau_0,\sigma_0)=p$ and $\partial_\tau X:=\frac{\partial X}{\partial \tau}$, $\partial_\sigma X:=\frac{\partial X}{\partial \sigma}$.   In doing so, one regards $\wedge^2 TM$ as part of an extended tangent bundle.  

For MVFs, the natural analogue of the Lie bracket is the Schouten-Nijenhuis bracket (SNB) \cite{Mar} \cite{Nij}.  The SNB of two multivector fields of degrees $k$ and $l$ is a multivector field of degree $k+l-1$.  With the SNB in the role of the Lie bracket, the idea of generalized geometry suggests that we must consider the full exterior bundle $\wedge^\bullet TM$ in the role of the tangent bundle as opposed to just $TM\oplus \wedge^2TM$.  Our basic ``philosophy" then is to treat $\wedge^k T_pM$ ($p\in M$, $k\ge 2$) as part of an extended tangent space on $M$.  Hence, a $k$-vector $v \in \wedge^k T_pM$ should be regarded as a new kind of tangent vector.   If we apply this viewpoint to the problem of covariant differentiation of MVFs along MVFs, one would expect $X\wedge Y$ and $Y\wedge Z$ to play a more explict role on the right side of (\ref{BPHPropertyA}) and (\ref{BPHPropertyB}) respectively.  However, for this to be true, conditions (\ref{BPHPropertyA}) and (\ref{BPHPropertyB}) must be relaxed; the upshot of this is the notion of a higher affine connection.   

A natural question in all this is where do higher connections actually arise.  More specifically, what framework requires the notion of \textit{non-induced} higher connections (i.e., the higher connecitons which do not satisfy (\ref{BPHPropertyA}) or (\ref{BPHPropertyB}))?  In this paper, we propose a solution to this question.  The answer comes by equipping $\wedge^\bullet TM$ with a smooth bilinear form $\eta$ which is associative in the sense that
\begin{equation}
\label{FrobeniusIntro}
\eta(x\wedge y,z)=\eta(x,y\wedge z)
\end{equation}
for all $x,y,z\in \wedge^\bullet T_pM$, $p\in M$.   These associative bilinear forms are shown to be in one to one correspondence with the space of differential forms: $\Omega^\bullet(M):=\bigoplus_{i}\Omega^i(M)$. This fact allows one to define a covariant derivative of $\eta$ with respect to a higher connection.   Naturally, one would like to find a higher connection for which 
\begin{equation}
\label{IntroParallel}
\nabla\eta \equiv 0.  
\end{equation}
A direct calculation shows that, in general, the induced higher connections (i.e., the ones that satisfy both (\ref{BPHPropertyA}) and (\ref{BPHPropertyB})) are incapable of satisfying (\ref{IntroParallel}).  By requiring $\nabla\eta\equiv 0$, the notion of non-induced higher connections becomes a necessary one.  Furthermore, since non-induced higher connections arise from associative bilinear forms on $\wedge^\bullet TM$, and the latter is equivalent to a set of differential forms of various degrees, this viewpoint also provides a way of linking higher connections to multisymplectic geometry \cite{CIL} \cite{Rog}.

The rest of the paper is organized as follows.  In section $2$, we review the basic machinery of MVFs and set up the notation we will use for the rest of the paper.  In section $3$, we introduce the notion of higher affine connections and prove a classification theorem for them (see Theorem \ref{HigherConnectionData}).  In addition, the notion of \textit{higher torsion} is also introduced.  In section $4$, we define the covariant derivative of differential forms along MVFs in terms of higher connections and examine the properties of this construction.   In section $5$, we relate non-induced higher connections to associative bilinear forms on the full exterior bundle $\wedge^\bullet TM$.  Finally, in section $6$, we conlcude the paper with some closing remarks and directions for future work.  

\section{Preliminaries}
\subsection{Some Multilinear Algebra}
\noindent In this brief section, we recall some basic results from multilinear algebra.  Throughout this section, $V$ is a finite dimensional vector space over $\mathbb{R}$ of dimension $m>0$.   
\begin{definition}
A $k$-vector $v\in \wedge^k V$ is \textit{decomposable} if it can be expressed as $v=v_1\wedge \cdots\wedge v_k$ for some $v_i\in V$, $i=1,\dots, k$.
\end{definition}
\begin{theorem}
\label{UniversalPropExterior}
Let $V^{\times k}:=V\times \cdots \times V$ ($k$ times) and let $U$ be a finite dimensional vector space over $\mathbb{R}$.  For any alternating multilinear map, $\varphi:V^{\times k}\rightarrow U$, there exists a unique linear map $\widetilde{\varphi}: \wedge^k V\rightarrow U$ such that 
\begin{equation}
\nonumber
\widetilde{\varphi}(v_1\wedge \cdots \wedge v_k)=\varphi(v_1,\dots, v_k)
\end{equation}
for all $v_1,\dots, v_k\in V$.  
\end{theorem}
\begin{proof}
Let $V^{\otimes k}:=V\otimes \cdots \otimes V$ ($k$ times).  Since $\varphi$ is multilinear, the universal property of the tensor product gives a unique linear map $\overline{\varphi}: V^{\otimes k}\rightarrow U$ such that $\overline{\varphi}(v_1\otimes \cdots \otimes v_k)=\varphi(v_1,\dots, v_k)$.  Let $I\subset V^{\otimes k}$ be the space spanned by elements of the form $v_1\otimes \cdots\otimes v_k$ where $v_1,\dots, v_k\in V$ and $v_i=v_j$ for some $i\neq j$.  Since $\varphi$ is alternating, we have $\overline{\varphi}|_I=0$.  Hence, $\overline{\varphi}$ induces a linear map from $\wedge^k V:=V^{\otimes k}/I$ to $U$ which satisfies 
\begin{equation}
\nonumber
\widetilde{\varphi}(v_1\wedge \cdots \wedge v_k)=\overline{\varphi}(v_1\otimes \cdots \otimes v_k)=\varphi(v_1,\dots, v_k).
\end{equation}
Since $\widetilde{\varphi}$ is linear and $\wedge^k V$ is spanned by decomposable $k$-vectors, $\widetilde{\varphi}$ is necessarily unique.  
\end{proof}
\begin{proposition}
\label{Multilinear1}
Let $W$ and $W'$ be any $k$-dimensional subspaces of $V$ and let $\{w_i\}$ and $\{w'_i\}$ be any bases on $W$ and $W'$ respectively.  Then
\begin{equation}
\nonumber
w_1\wedge \cdots \wedge w_k=\lambda w_1'\wedge \cdots \wedge w_k'
\end{equation}
for some $\lambda \in \mathbb{R}$ iff $W=W'$.  
\end{proposition}
\begin{proof}
Let $w:=w_1\wedge \cdots \wedge w_k$ and $w':=w_1'\wedge \cdots \wedge w_k'$.\\
($\Rightarrow$)  Suppose  $w=\lambda w'$ for some $\lambda\in \mathbb{R}$.  If $k=\dim V$, then $W=W'=V$.   Now suppose that $k<\dim V$.  If $W\neq W'$, then there exists a nonzero $w_{k+1}'\in W'$ which is not in $W$.   Since $w_{k+1}'$ is a linear combination of $\{w_1',\dots, w_k'\}$, we have $w'\wedge w_{k+1}'=0$.  On the other hand, $w\wedge w_{k+1}'\neq 0$.  This contradicts the assumption that $w=\lambda w'$. Hence, $W=W'$.  

($\Leftarrow$) Suppose $W=W'$.  Then $\{w_i\}$ can be expressed as a linear combination of $\{w_i'\}$:
\begin{equation}
\nonumber
w_j=\sum_{i=1}^k a_{ij} w_i'.
\end{equation}
A direct calculation shows that $w=\det(a_{ij})w'$.
\end{proof}

\noindent Proposition \ref{Multilinear1} shows that any subspace $W\subset V$ of dimension $k$ is represented by a decomposable $k$-vector which is unique up to a multiplicative constant.  This fact immediately implies the following:
\begin{corollary}
\label{Multilinear1Cor}
 Let $V$ be a finite dimensional vector space.  The mapping $\varphi: G(k,V)\rightarrow \mathbb{P}(\wedge^k V)$ given by 
\begin{equation}
W=\langle w_1,\dots,w_k\rangle \mapsto [w_1\wedge \cdots \wedge w_k],
\end{equation}
is well-defined and injective.  Here, $G(k,V)$ is the Grassmanian, the set of all $k$-dimensional subspaces of $V$ and $\mathbb{P}(\wedge^k V)$ is the projectivization of $\wedge^k V$. 
\end{corollary}
\begin{remark}
 The map $\varphi$ in Corollary \ref{Multilinear1Cor} is called the \textit{Pl\"{u}cker embedding}.
\end{remark}

\begin{lemma}
\label{Multilinear4}
Let $W\subset V$ be a subspace of $V$ of dimension $k>0$.  Let $w$ be any decomposable $k$-vector which represents $W$.  For any $x\in \wedge^m V$, there exists a decomposable $(m-k)$-vector $u$ such that $x=u\wedge w$.
\end{lemma}
\begin{proof}
Let $w=w_1\wedge\cdots\wedge w_k$ be any decomposable $k$-vector which represents $W$.  If $k=m$, then $W=V$ and any $x\in \wedge^m V$ is of the form $x=\lambda w$ for some $\lambda\in \mathbb{R}$.  In this case, we just take $u=\lambda\in \wedge^0 V:=\mathbb{R}$.  Now suppose that $k<m$.  By Proposition \ref{Multilinear1}, $\{w_i\}$ is a basis for $W$.  Extend $\{w_i\}_{i=1}^k$ to a basis on $V$: $\{w_1,\dots, w_k, e_1,\dots,e_{m-k}\}$.  Let $e=e_1\wedge \cdots \wedge e_{m-k}$.  Then $e\wedge w$ generates $\wedge^m V$.  Consequently, any $x\in \wedge^m V$ is of the form $x=\lambda e\wedge w$ for some $\lambda\in \mathbb{R}$.  Setting $u=\lambda e$ proves the lemma.
\end{proof}

\begin{proposition}
\label{Multilinear5}
Let $W$ and $W'$ be any two subspaces of $V$ of dimensions $k>0$ and $k'>0$ respectively.   Let $w$ and $w'$ be any decomposable $k$ and $k'$-vectors which represent $W$ and $W'$ repsectively.  Then $\dim W\cap W'>0$ iff $w\wedge w'=0$.
\end{proposition}
\begin{proof}
($\Rightarrow$) Suppose $l:=\dim W\cap W'>0$.  Let $w''$ be the decomposable $l$-vector which represents $W\cap W'$.  Since $W\cap W'$ is a subspace of both $W$ and $W'$, Lemma \ref{Multilinear4} implies that there exists a decomposable $(k-l)$-vector $u$ and a decomposable $(k'-l)$-vector $u'$ such that $w=u\wedge w''$ and $w'=u'\wedge w''$.  This implies that $w\wedge w'=0$.

($\Leftarrow$) Suppose $w\wedge w'=0$.  Since $w$ and $w'$ are decomposable, $w$ and $w'$ can be expressed as 
\begin{equation}
\nonumber
w=w_1\wedge \cdots \wedge w_k,\hspace*{0.1in} w'=w_1'\wedge \cdots \wedge w_{k'}'
\end{equation}
where $\{w_i\}$ and $\{w_i'\}$ are bases of  $W$ and $W'$ respectively.  The subspace $W+W'$ is spanned by $\{w_1,\dots,w_k,w_1',\dots, w'_{k'}\}$.  Since $w\wedge w'=0$, the aforementioned set cannot be linearly independent.  
Hence, $\dim (W+W')<\dim W+\dim W'$.  Since
\begin{equation}
\nonumber
\dim W\cap W' = \dim W+\dim W'-\dim(W+W'),
\end{equation}
we have $\dim W\cap W'>0$.
\end{proof}

\begin{proposition}
\label{Multilinear3}
Let $\{v_i^{(k)}\}_{i=1}^t$ be any linearly indpendent set of $k$-vectors in $\wedge^k V$ (not necessarily decomposable).  Then there exists an $(m-k)$-vector $u\in \wedge^{m-k}V$ (not necessarily decomposable) such that 
\begin{equation}
\nonumber
u\wedge v_1^{(k)}\neq 0,\hspace*{0.1in}u\wedge v_i^{(k)}=0~\mbox{for} ~i=2,\dots, t
\end{equation}
\end{proposition}
\begin{proof}
Fix a basis $\{e_1,\dots, e_m\}$ on $V$ and set $e:=e_1\wedge \cdots \wedge e_m$.  For any $u\in \wedge^{m-k} V$ and $v\in \wedge^k V$, define $\lambda_u(v)\in \mathbb{R}$ by
\begin{equation}
\nonumber
u\wedge v=\lambda_u(v)e.
\end{equation}
It's easy to see that $\lambda_u\in (\wedge^k V)^\ast$ and that 
\begin{equation}
\nonumber
u\mapsto \lambda_u,~\hspace*{0.1in} u\in \mbox{$\wedge^{m-k}V$}
\end{equation}
defines a linear map from $\wedge^{m-k}V$ to $(\wedge^k V)^\ast$.  We will now show that the aforementioned map is an isomorphism.  To do this, let
\begin{equation}
e^{(k)}_I:=e_{i_1}\wedge \cdots \wedge e_{i_k}
\end{equation}
where $I=\{i_1,\dots, i_k\}$ and $1\le i_1<\cdots <i_k\le m$.  Then $\{e^{(k)}_I\}_I$ is a basis on $\wedge^kV$.  Similarly, $\{e^{(m-k)}_J\}_J$ is a basis on $\wedge^{m-k} V$.  Then 
\begin{equation}
\nonumber
\lambda_{e^{(m-k)}_J}(e^{(k)}_I)=\pm 1\hspace*{0.1in} \mbox{iff $I\cap J=\emptyset$}
\end{equation}
and zero otherwise. Hence, up to a sign, the set $\{\lambda_{e^{(m-k)}_J}\}_J$ is the dual basis of $\{e^{(k)}_I\}_I$.  This establishes the isomorphism.

To prove the proposition, note that since $\{v_i^{(k)}\}_{i=1}^t$ is linearly independent, there exists an element $\varphi\in (\wedge^k V)^\ast$ such that $\varphi(v_1^{(k)})=1$ and $\varphi(v_i^{(k)})=0$ for $i>1$.  Consequently, there exists some $u\in \wedge^{m-k}V$ such that $\lambda_u=\varphi$.  From this, we have
\begin{equation}
\nonumber
u\wedge v_1^{(k)}=\lambda_u(v_1^{(k)}) e = \varphi(v_1^{(k)})e=e\neq 0
\end{equation}
and 
\begin{equation}
\nonumber
u\wedge v_i^{(k)}=\lambda_u(v_i^{(k)}) e = \varphi(v_i^{(k)})e=0,\hspace*{0.1in}  i>1
\end{equation}
This completes the proof.
\end{proof}
\begin{proposition}
\label{VectorCovectorProp}
$(\wedge^kV)^\ast$ and $\wedge^k V^\ast$ are naturally isomorphic.
\end{proposition}
\begin{proof}
Let $\omega:=\omega^1\wedge \cdots \wedge \omega^k$ and $v:=v_1\wedge\cdots \wedge v_k$ be decomposable $k$-vectors in $\wedge^k V^\ast$ and $\wedge^k V$ respectively.  Define
\begin{equation}
\label{VectorCovector}
\varphi_\omega(v):=\det(\omega^i(v_j)).
\end{equation}
The fact that the determinant is an alternating multilinear map with respect to the $v_j$'s (and the $\omega^i$'s) implies that $\varphi_\omega\in (\wedge^kV)^\ast$, and that (\ref{VectorCovector}) extends to a linear map
\begin{equation}
\nonumber
\varphi: \mbox{$\wedge^k V^\ast$}\rightarrow \mbox{$(\wedge^kV)^\ast$},\hspace*{0.1in} \omega\mapsto \varphi_\omega.
\end{equation}
 To see that this is an isomorphism, let $e_1,\dots, e_m$ be a basis on $V$ and $\phi^1,\dots, \phi^m$ the dual basis.  Let 
\begin{equation}
\mc{I}_k:=\{(i_1,\dots, i_k)~|~1\le i_1<\cdots <i_k\le m\},
\end{equation}
and for $I=(i_1,\dots, i_k)\in\mc{I}_k$, let $e_I:=e_{i_1}\wedge \cdots \wedge e_{i_k}$ and $\phi^I:=\phi^{i_1}\wedge \cdots \wedge \phi^{i_k}$. Then $\{e_I\}$ and $\{\phi^I\}$ are bases on $\wedge^k V$ and $\wedge^k V^\ast$ respectively.  It follows easily from (\ref{VectorCovector}) that $\{\varphi_{\phi^I}\}$ is the dual basis of  $\{e_I\}$.  Hence, $\varphi$ is an isomorphism.
\end{proof}

\subsection{Multivector Fields}
\noindent Let $M$ be a smooth manifold.  To simplify notation, we set $A^k(M):=\Gamma(\wedge^k TM)$ and $\mc{A}(M):=\bigoplus_{k=0}^\infty A^k(M)$, where for a vector bundle $E\rightarrow M$, $\Gamma(E)$ denotes the space of sections of $E$.  The space of $k$-forms on $M$ is denoted as $\Omega^k(M)$.  For convenience, we also set $A^k(M)=0$ and $\Omega^k(M)=0$ for $k<0$.  
\begin{definition}
Let $k\in \mathbb{N}$.  A \textit{multiderivation of degree $k$} (or $k$-\textit{derivation}) on a manifold $M$ is a $k$-linear map
\begin{equation}
\nonumber
\varphi: C^{\infty}(M)\times \cdots \times C^\infty(M)\rightarrow C^\infty(M)
\end{equation}
over $\mathbb{R}$, which is totally antisymmeric and a derivation of $C^\infty(M)$ in each of its arguments, i.e.
\begin{itemize}
\item[(i)] $\varphi(f_{\sigma(1)},\dots,f_{\sigma(k)})=\mbox{sgn}(\sigma)\varphi(f_1,\dots, f_k)$ $\forall~\sigma\in \mathfrak{S}_k$
\item[(ii)] $\varphi(f_1g,f_2,\dots, f_k)=\varphi(f_1,\dots,f_k)g+f_1\varphi(g,f_2,\dots, f_k)$
\end{itemize}
for all $f_i,~g\in C^\infty(M)$.
\end{definition}

\noindent The next two results are well known in differential geometry and we state them without proof.
\begin{proposition}
Every $k$-vector field $X\in A^k(M)$ defines a $k$-derivation via
\begin{equation}
\label{kVectorField1}
X(f_1,\dots,f_k):=(df_1\wedge\cdots\wedge df_k)(X).
\end{equation}
\end{proposition}

\begin{proposition}
There is a one-one correspondence between the space of $k$-derivations and the space of $k$-vector fields.  Specifically, every $k$-derivation $\varphi$ is given by
\begin{equation}
\nonumber
\varphi(f_1,\dots, f_k)=X(f_1,\dots, f_k),\hspace*{0.1in} f_i\in C^\infty(M),~i=1,\dots, k,
\end{equation}
for some unique $X\in A^k(M)$.  
\end{proposition}

\begin{definition}
\label{SchoutenBracketMVF}
The \textit{Schouten-Nijenhuis bracket} of multivector fields is the unique $\mathbb{R}$-bilinear map\footnote{see Proposition 3.1 of \cite{Mar}}
\begin{equation}
\nonumber
[\cdot,\cdot]: A^k(M)\times A^l(M)\rightarrow A^{k+l-1}(M)
\end{equation}
which satisfies the following conditions:
\begin{itemize}
\item[(i)] For $f,g\in C^\infty(M)$, $[f,g]=0$
\item[(ii)] For $X\in A^1(M)$, $Q\in \mc{A}(M)$, $[X,Q]=L_X Q$, the Lie derivative of $Q$ with respect to $X$
\item[(iii)] For $P\in A^p(M)$, $Q\in A^q(M)$, $[P,Q]=-(-1)^{(p-1)(q-1)}[Q,P]$
\item[(iv)]  $ad_P:=[P,\cdot]$ is a derivation of degree $p-1$ for $P\in A^p(M)$ of the exterior product on $\mc{A}(M)$, that is,
\begin{equation}
\nonumber
ad_P(Q\wedge R)=ad_P(Q)\wedge R+(-1)^{(p-1)q}Q\wedge ad_P(R)
\end{equation}
for $Q\in A^q(M)$, $R\in \mc{A}(M)$.
\end{itemize}
\end{definition}
\noindent Definition \ref{SchoutenBracketMVF} implies that \cite{Mar}
\begin{align}
\nonumber
(-1)^{(p-1)(r-1)}&[P,[Q,R]]+(-1)^{(q-1)(p-1)}[Q,[R,P]]\\
\nonumber
&+(-1)^{(r-1)(q-1)}[R,[P,Q]]=0
\end{align}
for $P\in A^p(M)$, $Q\in A^q(M)$, and $R\in A^r(M)$.  This together with (iii) of Definition \ref{SchoutenBracketMVF} shows that $(\mc{A}(M),[\cdot,\cdot])$ is a graded Lie algebra if $\deg ~A^p(M):=p-1$.

For $X_1,X_2\dots, X_p,Y_1,\dots, Y_q\in A^1(M)$, the Schouten-Nijenhuis bracket is given explicitly by
\begin{align}
\nonumber
[X_1\wedge &\dots \wedge X_p,Y_1\wedge\cdots \wedge Y_q]\\
\nonumber
&= \sum_{i,j}(-1)^{i+j}[X_i,Y_j]\wedge X_1\wedge \cdots \wedge \hat{X}_i\wedge\cdots\wedge X_p\wedge Y_1\wedge \cdots\wedge \hat{Y}_j\wedge\cdots \wedge Y_q.
\end{align}

\begin{definition}
\label{InteriorProductDef}
The interior product of a smooth function $f$ with a $k$-derivation $X\in A^k(M)$ ($k\ge 1$) is the $k-1$-derviation $i_fX$ defined by 
\begin{equation}
\nonumber
i_{f}X(g_1,\dots, g_{k-1}):=X(f,g_1,\dots, g_{k-1}),
\end{equation}
for $g_1,\dots, g_{k-1}\in C^\infty(M)$.  For $g\in A^0(M):=C^\infty(M)$, $i_fg:=0$.
\end{definition}

\begin{proposition}
\label{InteriorProductProp}
Let $f,g\in C^\infty(M)$, $X\in A^k(M)$, and $Y,Y'\in \mc{A}(M)$.  Then
\begin{itemize}
\item[(i)] $i_{f+g} X=i_fX+i_gX$
\item[(ii)] $i_f(Y+Y')=i_f Y+i_fY'$
\item[(iii)] $i_{fg} X=g i_f X+ fi_g X$
\item[(iv)] $i_f(X\wedge Y)=(i_f X)\wedge Y+(-1)^kX\wedge (i_f Y)$
\item[(v)] $[X,f]=(-1)^{k-1}i_fX$ 
\item[(vi)] $[fX,Y]=f[X,Y]-X\wedge i_f Y$
\end{itemize}
\end{proposition}
\begin{proof}
(i) and (ii) are immediate.

Note that (iii)-(v) is satisfied for the case when $X$ is a smooth function.  We now prove (iii)-(v) for the case when $X\in A^k(M)$ with $k\ge 1$.  

(iii) is a direct consequence of the fact that $X$ is a derivation of $C^\infty(M)$ in each of its arguments.  Specifically,
\begin{align}
\nonumber
i_{fg}X&:=X(fg,\cdot,\dots, \cdot)\\
\nonumber
&=X(f,\cdot, \dots, \cdot)g+fX(g,\cdot,\dots,\cdot)\\
\nonumber
&=gi_f X+fi_g X.
\end{align}

For (iv), take $Y\in A^{l}(M)$ without loss of generality and set $g_1:=f$.  Then
\begin{align}
\nonumber
i_f (X\wedge Y)(g_2,\dots, g_{k+l})&=(X\wedge Y)(g_1,g_2,\dots,g_{k+l})\\
\label{WedgeSum1}
&=\sum_{\sigma\in S(k,l)} \epsilon(\sigma)X(g_{\sigma(1)},\dots, g_{\sigma(k)})Y(g_{\sigma(k+1)},\dots,g_{\sigma(k+l)}), 
\end{align}
where $\sigma\in S(k,l)\subset \mathfrak{S}_{k+l}$ is a shuffle permutation, i.e., $\sigma$ satisfies
\begin{equation}
\nonumber
\sigma(1)<\cdots <\sigma(k),\hspace*{0.2in} \sigma(k+1)<\cdots <\sigma(k+l)
\end{equation}
and $\epsilon(\sigma)=+1$ $(-1)$ if $\sigma$ is even (odd).   (\ref{WedgeSum1}) can be decomposed as
\begin{align}
\nonumber
&\sum_{\sigma\in S(k,l),~\sigma(1)=1} \epsilon(\sigma)X(g_{\sigma(1)},\dots, g_{\sigma(k)})Y(g_{\sigma(k+1)},\dots,g_{\sigma(k+l)})\\
\nonumber
&+\sum_{\sigma\in S(k,l),~\sigma(k+1)=1} \epsilon(\sigma)X(g_{\sigma(1)},\dots, g_{\sigma(k)})Y(g_{\sigma(k+1)},\dots,g_{\sigma(k+l)})\\
\nonumber
&=\sum_{\sigma\in S(k-1,l)} \epsilon(\sigma)i_fX(\tilde{g}_{\sigma(1)},\dots, \tilde{g}_{\sigma(k-1)})Y(\tilde{g}_{\sigma(k)},\dots, \tilde{g}_{\sigma(l+k-1)})\\
\label{WedgeSum2}
&+(-1)^{kl}\sum_{\sigma\in S(l-1,k)} \epsilon(\sigma)i_fY(\tilde{g}_{\sigma(1)},\dots, \tilde{g}_{\sigma(l-1)})X(\tilde{g}_{\sigma(l)},\dots, \tilde{g}_{\sigma(k+l-1)}),
\end{align}
where $\tilde{g}_i=g_{i+1}$ for $i=1,\dots,k+l-1$.  (\ref{WedgeSum2}) can then be rewritten as
\begin{align}
\nonumber
&(i_f X\wedge Y)(g_2,\dots, g_{k+l})+(-1)^{kl}(i_f Y\wedge X)(g_2,\dots, g_{k+l})\\
\nonumber
&=(i_f X\wedge Y)(g_2,\dots, g_{k+l})+(-1)^{kl+k(l-1)}(X\wedge i_f Y)(g_2,\dots, g_{k+l})\\
\nonumber
&=(i_f X\wedge Y)(g_2,\dots, g_{k+l})+(-1)^{k}(X\wedge i_f Y)(g_2,\dots, g_{k+l})
\end{align}

For (v), let $X\in A^k(M)$. For $k=0$, the result follows from Definition \ref{SchoutenBracketMVF} and Definition \ref{InteriorProductDef}.  Now consider the case when $k\ge 1$.  Condition (iv) of  Definition \ref{SchoutenBracketMVF} implies that the Schouten-Nijenhuis bracket is local in nature.  Since $X$ is locally a finite sum of decomposable terms, it sufficies to  prove (v) of Proposition \ref{InteriorProductProp} for the case when $X$ is decomposable, i.e. $X=X_1\wedge \cdots \wedge X_k$ where $X_i\in A^1(M)$ for $i=1,\dots, k$.  For $k=1$, $X\in A^1(M)$ and 
\begin{equation}
\nonumber
[X,f] = L_X f=Xf=i_fX=(-1)^{k-1}i_fX,
\end{equation}
by (ii) of Definition \ref{SchoutenBracketMVF}.  We now prove (v) by induction on $k$.  Suppose then that (v) holds for $k$ (where $k\ge 1$) and let $X_{k+1}\in A^1(M)$.  By  (iii) and (iv) of Definition \ref{SchoutenBracketMVF}, we have
\begin{align}
\nonumber
[X\wedge X_{k+1},f]&=(-1)^{k+1}[f,X\wedge X_{k+1}]\\
\nonumber
&=(-1)^{k+1}([f,X]\wedge X_{k+1}+(-1)^{k} X\wedge [f,X_{k+1}])\\
\nonumber
&=(-1)^{k+1}((-1)^{k}[X,f]\wedge X_{k+1}+(-1)^{k+1}X\wedge [X_{k+1},f])\\
\nonumber
&=(-1)^{k+1}(-i_f X\wedge X_{k+1}+(-1)^{k+1} X\wedge i_f X_{k+1})\\
\nonumber
&=(-1)^k(i_f X\wedge X_{k+1}+(-1)^k X\wedge i_f X_{k+1})\\
\nonumber
&=(-1)^ki_f (X\wedge X_{k+1}),
\end{align}
where we used the induction hypothesis in the fourth equality and (iv) of Proposition \ref{InteriorProductProp} in the sixth equality.  

For (vi), take $Y\in A^l(M)$ without loss of generality.  Then
\begin{align}
\nonumber
[fX,Y]&=[f\wedge X,Y]\\
\nonumber
&=-(-1)^{(k-1)(l-1)}[Y,f\wedge X]\\
\nonumber
&=-(-1)^{(k-1)(l-1)}([Y,f]\wedge X+f\wedge[Y,X])\\
\nonumber
&=-(-1)^{(k-1)(l-1)}((-1)^{(l-1)}i_fY\wedge X+f[Y,X])\\
\nonumber
&=f[X,Y]-(-1)^{(l-1)k} i_fY \wedge X\\
\nonumber
&=f[X,Y]-X\wedge i_fY,
\end{align}
where we used (iii) and (iv) of Definition \ref{SchoutenBracketMVF} in the second and third equalities respectively, and (v) of Proposition \ref{InteriorProductProp} was used in the fourth equality. This completes the proof. 
\end{proof}
\begin{definition}
Let $X\in A^k(M)$, $k\ge 1$.  The \textit{interior product} by $X$ \cite{Mar} is the $C^\infty(M)$-linear map $i_X: \Omega^l(M)\rightarrow \Omega^{l-k}(M)$ which is defined as follows:
\begin{itemize}
\item[(i)] for $\omega\in \Omega^l(X)$, $l> k$, $i_X\omega$ is given by
\begin{equation}
i_X\omega(Y):=\omega(X\wedge Y),\hspace*{0.2in} \forall ~Y\in A^{l-k}(M),
\end{equation}
\item[(ii)] for $\omega\in \Omega^l(X)$, $l= k$, $i_X\omega$ is given by
\begin{equation}
i_X\omega:=\omega(X),
\end{equation}
\item[(iii)] for $\omega\in \Omega^l(X)$, $l< k$, $i_X\omega:=0$.
\end{itemize}
\vspace*{0.1in}
For $f\in A^0(M):=C^\infty(M)$, $i_f\omega:=f\omega$ $~\forall~\omega\in \Omega^\bullet(M)$.
\end{definition}
\begin{remark}
For $X\in A^1(M)$, one can show that $i_X$ is a derivation of degree $-1$, that is,  
\begin{equation}
\label{InteriorProductDeg1}
i_X(\omega\wedge \eta)=(i_X\omega)\wedge \eta+(-1)^l\omega\wedge i_X\eta,
\end{equation}
for $\omega\in \Omega^l(M)$, $\eta\in \Omega^\bullet(M)$. 
\end{remark}

\begin{proposition}
\label{InteriorProduct1}
Let $X\in A^k(M)$ and $Y\in A^l(M)$.  Then
\begin{itemize}
\item[(i)] $i_{X\wedge Y}= i_Y\circ i_X$ 
\item[(ii)]  $ i_Y\circ i_X=(-1)^{kl}i_X\circ i_Y$
\end{itemize}
\end{proposition}
\begin{proof}
Let $\omega\in \Omega^p(M)$ with $p > k+l$ and let $Z\in A^{p-k-l}(M)$.  For (i), we have
\begin{align}
\nonumber
i_{X\wedge Y}\omega(Z)&:=\omega(X\wedge Y\wedge Z)\\
\nonumber
&=(i_X\omega)(Y\wedge Z)\\
\nonumber
&=i_Y(i_X\omega)(Z).
\end{align}
The case when $p=k+l$ is handled similarly.  

For (ii), we have
\begin{align}
\nonumber
i_Y\circ i_X = i_{X\wedge Y}=(-1)^{kl}i_{Y\wedge X}=(-1)^{kl}i_X\circ i_Y
\end{align}
where the first and third equality follows from part (i) of Proposition \ref{InteriorProduct1}.
\end{proof}

We conclude this section by recalling the Lie derivative of a differential form $\omega$ along a  MVF $X\in A^k(M)$ $(k\ge 1)$ \cite{FPR}:
\begin{equation}
\label{LieDerivativeMVF}
L_X\omega:=d i_X\omega - (-1)^k i_X d\omega.
\end{equation}
For $\omega\in \Omega^l(M)$, (\ref{LieDerivativeMVF}) implies that $L_X\omega\in \Omega^{l-k+1}(M)$.  The next result summarizes the properties of the Lie derivative of differential forms along MVFs:
\begin{proposition}\cite{FPR}
\label{LieDerivativeProp}
Let $X\in A^k(M)$, $Y\in A^l(M)$, and $\omega\in \Omega^\bullet(M)$.  Then
\begin{itemize}
\item[(i)] $dL_X\omega = (-1)^{k-1}L_X d\omega$
\item[(ii)] $i_{[X,Y]} \omega=(-1)^{(k-1)l}L_X i_Y\omega-i_Y L_X\omega$
\item[(iii)] $L_{[X,Y]}\omega=(-1)^{(k-1)(l-1)}L_X L_Y\omega-L_YL_X\omega$
\item[(iv)] $L_{X\wedge Y} \omega=(-1)^l i_Y L_X\omega+L_Y i_X\omega$
\end{itemize}
\end{proposition}
\begin{proof}
See Proposition A.3 of \cite{FPR}.
\end{proof}

\section{Higher Affine Connections}
\begin{definition}
\label{HCDef}
A \textit{higher affine connection} (or \textit{higher connection}) on $M$ is a map 
\begin{align}
\nonumber
&\nabla: \mc{A}(M)\times \mc{A}(M)\rightarrow  \mc{A}(M),\hspace*{0.2in}(X,Y)\mapsto \nabla_X Y
\end{align}
such that
\begin{itemize}
\item[(i)] $\nabla_{X} Y\in A^{k+l-1}(M)$ for $X\in A^k(M)$, $Y\in A^l(M)$
\item[(ii)] $\nabla_{fX+X'}Y=f\nabla_XY+\nabla_{X'} Y$ for $X,X',Y\in \mc{A}(M)$
\item[(iii)] $\nabla_X(Y+Y')=\nabla_XY+\nabla_X Y'$ for $X,Y,Y'\in\mc{A}(M)$
\item[(iv)] $\nabla_X f=[X,f]$ for $X\in A^k(M)$, $f\in C^\infty(M)$
\item[(v)] $\nabla_X fY=[X,f]\wedge Y+f\nabla_X Y$, for $X\in A^k(M)$, $f\in C^\infty(M)$, $Y\in \mc{A}(M)$
\item[(vi)] $\nabla_f X=0$ for $f\in C^\infty(M)$
\end{itemize}
\end{definition}

\begin{corollary}
\label{HCCor1}
Let $\nabla$be a higher connection on $M$.  Then 
\begin{itemize}
\item[(i)] $\nabla_X fY=(-1)^{k-1}i_f X\wedge Y+f\nabla_X Y$ for $X\in A^k(M)$, $Y\in \mc{A}(M)$, and $f\in C^\infty(M)$; in particular, $\nabla_X f = (-1)^{k-1}i_f X$
\item[(ii)] the restriction of $\nabla$ to $A^1(M)\times A^1(M)$ is an affine connection on $M$.
\end{itemize}
\end{corollary}
\begin{proof}
(i) of Corollary \ref{HCCor1} follows from  (iv) and (v) of Definition \ref{HCDef} and Proposition \ref{InteriorProductProp}-(v).  (ii) of  Corollary \ref{HCCor1}  follows from Definition \ref{HCDef} and (i) of Corollary \ref{HCCor1}. 
\end{proof}
\begin{remark}
It's imporant to stress that even when a higher connection is restricted to covariant differentiation along 1-vector fields, the result will not (in general) coincide with the usual extension of an affine connection to arbitrary tensor fields.  For example, if $\nabla$ is a higher connection and $X\in A^1(M)$, $Y\in A^l(M)$, and $Z\in \mc{A}(M)$, then, in general, 
\begin{equation}
\label{compareAffine}
\nabla_{X} (Y\wedge Z)\neq (\nabla_X Y)\wedge Z+Y\wedge \nabla_X Z.
\end{equation}
The right side of (\ref{compareAffine}) is exactly how an affine connection on $TM$ would operate on $Y\wedge Z$.  The fact that higher connections do not do this (in general) is a consequence of the fact that higher connections attempt to put $k$-vector fields and 1-vector fields on a more equal footing.  Hence, a higher connection will view a $k$-vector field as being   ``indivisible" in some sense.  Consequently, $\nabla_X(Y\wedge Z)$ will depend not only on $Y$ and $Z$, but also on $Y\wedge Z$.  
\end{remark}

\begin{proposition}
\label{ExistenceHC}
There is a one to one correspondence between affine connections on $TM$ and higher connections satisfying
\begin{align}
\label{ExistenceHC1}
\nabla_{X\wedge Y} Z&=X\wedge \nabla_Y Z+(-1)^{kl} Y\wedge \nabla_X Z\\
\label{ExistenceHC2}
\nabla_X(Y\wedge Z)&=(\nabla_X Y)\wedge Z+(-1)^{(k-1)l}Y\wedge \nabla_X Z,
\end{align}
for $X\in A^k(M)$, $Y\in A^l(M)$, and $Z\in \mc{A}(M)$. 
\end{proposition}
\begin{proof}
Let $\mc{H}_0$ be the set of all higher connections satisfying (\ref{ExistenceHC1}) and (\ref{ExistenceHC2}) and let $\mc{C}$ be the set of all affine connections on $TM$.  Let $\varphi: \mc{H}_0\rightarrow \mc{C}$ be the map which sends $\nabla\in \mc{H}_0$ to its restriction $\nabla|_{A^1(M)\times A^1(M)}\in \mc{C}$.  To see that this map is injective, let $\nabla$ be any higher connection which satisfies (\ref{ExistenceHC1}) and let $X=X_1\wedge \cdots \wedge X_k$ be a decomposable $k$-vector field.  Then for $Y\in \mc{A}(M)$, we have
\begin{equation}
\label{ExplicitForm1}
\nabla_X Y=\sum_{j=1}^k(-1)^{k-j}X_1\wedge \cdots \wedge \wh{X}_j\wedge \cdots \wedge X_k \wedge \nabla_{X_j} Y,
\end{equation}
where $\wh{X}_j$ denotes the omission of $X_j$.  On the other hand, if $\nabla$ satisfies (\ref{ExistenceHC2}) and $Y=Y_1\wedge \cdots \wedge Y_l$ is a decomposable $l$-vector field, then for $X\in \mc{A}(M)$, we have
\begin{equation}
\label{ExplicitForm2}
\nabla_X Y= \sum_{j=1}^l (-1)^{j-1}(\nabla_X Y_j)\wedge Y_1\wedge \cdots \wedge \widehat{Y_j}\wedge \cdots \wedge Y_l.
\end{equation}
Equations (\ref{ExplicitForm1}) and (\ref{ExplicitForm2}) imply that if $\nabla$ satisfies \textit{both}  (\ref{ExistenceHC1}) and  (\ref{ExistenceHC2}), then $\nabla$ is completely determined as a higher connection by its restriction to $A^1(M)\times A^1(M)$, which is simply an affine connection on $TM$.  Consequently, any two higher connections in $\mc{H}_0$ which agree on $A^1(M)\times A^1(M)$ must be the same.  This proves that $\varphi$ is injective.

To see that $\varphi$ is surjective, let $\widetilde{\nabla}$ be any affine connection on $TM$.   We now extend $\widetilde{\nabla}$ to a higher connection $\nabla'$ as follows.  For $X\in A^k(M)$ and $f\in C^\infty(M)$, define $\nabla'_f\equiv 0$ and $\nabla'_X f:=[X,f]$.  To extend $\nabla'$ to $A^k(M)\times A^l(M)$ for $k,l>0$, let $X=X_1\wedge \cdots \wedge X_k\in A^k(M)$ be a decomposable $k$-vector field and let $Y\in A^l(M)$ be any $l$-vector field.  We define $\nabla'_XY$ via
\begin{equation}
\label{ExplicitForm3A}
\nabla'_X Y=\sum_{j=1}^k (-1)^{k-j} X_1\wedge \cdots \wedge\widehat{X_j}\wedge \cdots \wedge X_k\wedge \widetilde{\nabla}_{X_j} Y,
\end{equation}
where $\widetilde{\nabla}_{X_j} Y$ is defined in the usual way and $\widehat{X_j}$ denotes omission as usual.  With $Y$ fixed, let
\begin{equation}
\nonumber
\rho(X_1,\dots, X_k):=\nabla'_XY.
\end{equation}
A direct calculation shows that $\rho$ is an alternating $C^\infty(M)$-multilinear map.  This implies that (\ref{ExplicitForm3A}) extends to all $X\in A^k(M)$ by $C^\infty(M)$-linearity.  Using Lemma \ref{SchoutenLemma} from section 4, one can show that 
\begin{equation}
\nonumber
\nabla'_Xf Y=[X,f]\wedge Y+f\nabla'_X Y,
\end{equation}
for $f\in C^\infty(M)$.  The other axioms of Definition \ref{HCDef} are are easily verified.  Hence, $\nabla'$ is a higher connection. Another straightforward calculation shows that $\nabla'$ also satisfies (\ref{InducedHC1}) and (\ref{InducedHC2}).   Furthermore, (\ref{ExplicitForm3A}) implies that $\varphi(\nabla')=\widetilde{\nabla}$.  This completes the proof.
\end{proof}

\noindent Proposition \ref{ExistenceHC} motivates the following definition.
\begin{definition}
\label{InducedHC}
A higher connection $\nabla$ is called \textit{induced} if it satisfies
\begin{align}
\label{InducedHC1}
\nabla_{X\wedge Y}Z&=X\wedge \nabla_Y Z+(-1)^{kl}Y\wedge \nabla_X Z\\
\label{InducedHC2}
\nabla_X(Y\wedge Z)&=(\nabla_X Y)\wedge Z+(-1)^{(k-1)l}Y\wedge \nabla_X Z
\end{align}
for $X\in A^k(M)$, $Y\in A^l(M)$, and $Z\in \mc{A}(M)$. 
\end{definition}

\begin{remark}
For an affine connection $\widetilde{\nabla}$ on $TM$, we will use the same symbol $\widetilde{\nabla}$ to denote the induced higher connection associated to $\widetilde{\nabla}$ that was constructed in the proof of Proposition \ref{ExistenceHC}.  
\end{remark}

\begin{remark}
The higher connection given by Proposition \ref{ExistenceHC} is equivalent to the covariant derivative introduced in \cite{BPH}.  To see this, let $\nabla$ be a higher connection given by Proposition \ref{ExistenceHC} and define $\nabla'$ by
\begin{equation}
\nonumber
{\nabla'}_XY:=(-1)^{k-1} \nabla_X Y.
\end{equation}
for $X\in A^k(M)$, $Y\in \mc{A}(M)$.  Then a direct calculation shows
\begin{align}
\nonumber
{\nabla'}_{X\wedge Y}Z&=(-1)^k X\wedge {\nabla'}_Y Z+(-1)^{l(k-1)}Y\wedge {\nabla'}_X Z\\
\nonumber
{\nabla'}_{X}(Y\wedge Z)&=({\nabla'}_X Y)\wedge Z+(-1)^{l(k-1)}Y\wedge {\nabla'}_X Z
\end{align}
which is precisely condition (\ref{BPHPropertyA}) and (\ref{BPHPropertyB}) from \cite{BPH}. 
\end{remark}

\begin{lemma}
\label{DifferenceHC}
Let $\nabla$ and ${\nabla}'$ be two higher connections on $M$ and let $F:\mc{A}(M)\times \mc{A}(M)\rightarrow \mc{A}(M)$ be given by $F(X,Y):=\nabla_XY-{\nabla'}_X Y$.  Then $F$ is $C^\infty(M)$-linear in $X$ and $Y$.  In particular, $F^{k,l}:=F|_{A^k(M)\times A^l(M)}$ is a section of the bundle $\wedge^{k+l-1} TM\otimes \wedge^k T^\ast M\otimes \wedge^l T^\ast M$ for $k,l>0$ with $k+l-1\le n:=\dim M$.
\end{lemma}
\begin{proof}
Let $X\in A^k(M)$, $Y\in \mc{A}(M)$, and $h\in C^\infty(M)$.  Its clear that $F(hX,Y)=hF(X,Y)$.  We now show that $F$ is $C^\infty(M)$-linear in $Y$:
\begin{align}
\nonumber
F(X,hY)&=\nabla_{X}(hY)-{\nabla'}_{X}(hY)\\
\nonumber
&=[X,h]\wedge Y+h\nabla_X Y- [X,h]\wedge Y-h{\nabla'}_X Y\\
\nonumber
&=hF(X,Y).
\end{align}
\end{proof}

\begin{theorem}
\label{HigherConnectionData}
Let $\nabla$ be any higher connection on $M$.  Then there exists
\begin{itemize}
\item[(i)]  a unique affine connection $\widetilde{\nabla}$ on $TM$, and
\item[(ii)] a unique collection of sections $F^{k,l}$ of  the bundle 
\begin{equation}
\nonumber
E^{k,l}:=\mbox{$\wedge^{k+l-1} TM$}\otimes \mbox{$\wedge^k T^\ast M$}\otimes  \mbox{$\wedge^lT^\ast M$}
\end{equation}
for $k,l>0$ with $k+l-1\le n:=\dim M$  
\end{itemize}
such that $\forall~X\in A^k(M),~Y\in A^l(M)$ 
\begin{align}
\label{HCData1}
\nabla_X Y = \widetilde{\nabla}_X Y+F^{k,l}(X,Y),\hspace*{0.1in} 
\end{align}
where $F^{1,1}\equiv 0$ and $\widetilde{\nabla}_X Y$ in (\ref{HCData1}) is understood to be the higher connection induced by the affine connection $\widetilde{\nabla}$ according to  (\ref{ExistenceHC1}) and (\ref{ExistenceHC2}).  Conversely, any affine connection $\widetilde{\nabla}$ on $TM$ together with any collection of sections $F^{k,l}\in \Gamma(E^{k,l})$ for $k,l>0$ with $k+l-1\le n$ and $F^{1,1}\equiv 0$ determines a unique higher connection on $M$ which satisfies (\ref{HCData1}).  In particular, there is a bijection between the space of all higher connections and the set of all pairs of the form $(\widetilde{\nabla},\{F^{k,l}\})$, where $\widetilde{\nabla}$ is an affine connection on $TM$ and $F^{k,l}\in \Gamma(E^{k,l})$ for $k,l>0$ with $k+l-1\le n$ and $F^{1,1}\equiv 0$.
\end{theorem}

\begin{proof}
Let $\nabla$ be any higher connection on $M$ and let $\widetilde{\nabla}$ be the affine connection on $TM$ defined by $\widetilde{\nabla}_XY:=\nabla_XY$ for $X,Y\in A^1(M)$.  Extend $\widetilde{\nabla}$ to a higher connection on $M$ via  (\ref{ExistenceHC1}) and (\ref{ExistenceHC2}).  By Lemma \ref{DifferenceHC}, 
\begin{equation}
F^{k,l}:=(\nabla-\widetilde{\nabla})|_{A^k(M)\times A^l(M)}:A^k(M)\times A^l(M)\rightarrow A^{k+l-1}(M)
\end{equation}
is a section of $E^{k,l}:=\wedge^{k+l-1} TM\otimes \wedge^k T^\ast M\otimes \wedge^l T^\ast M$ for $k,l>0$ with $k+l-1\le n:=\dim M$.  (Note that $F^{1,1}\equiv 0$.)   Hence, we have
\begin{equation}
\nonumber
\nabla_X Y= \widetilde{\nabla}_X Y+F^{k,l}(X,Y),
\end{equation}
for $X\in A^k(M)$, $Y\in A^l(M)$, $k,l>0$.  To see that $(\widetilde{\nabla},\{F^{k,l}\})$ is unique, suppose that $(\widehat{\nabla},\{G^{k,l}\})$ is another such pair which satisfies (\ref{HCData1}).  Then for all $X,Y\in A^1(M)$, we have
\begin{equation}
\nonumber
\widetilde{\nabla}_XY=\nabla_XY=\widehat{\nabla}_XY.
\end{equation}
Hence, $\widetilde{\nabla}=\widehat{\nabla}$.  This fact together with (\ref{HCData1}) then implies that $F^{k,l}=G^{k,l}$ for all $k,l$.

Conversely, suppose $\widetilde{\nabla}$ is an affine connection on $TM$ (extended to a higher connection on $M$ via  (\ref{ExistenceHC1}) and (\ref{ExistenceHC2})) and $F^{k,l}$ is a section of the bundle $E^{k,l}$ for $k,l>0$ with $k+l-1\le n$ and $F^{1,1}\equiv 0$.  For $X\in A^k(M)$,  $Y\in A^l(M)$ with $k,l>0$ and $k+l-1\le n$, define $\nabla_XY$ according to (\ref{HCData1}) and set $\nabla_X f:=[X,f]$ and $\nabla_f :=0$ for $f\in C^\infty(M)$.  Its clear that $\nabla$ satisfies all the axioms of Definition \ref{HCDef} with the possible exception of axiom (v).  To verify axiom (v), let 
$X\in A^k(M)$, $Y\in A^l(M)$, and $f\in C^\infty(M)$ with $k,l>0$ and $k+l-1\le n$.  Then
\begin{align}
\nonumber
\nabla_X fY&=\widetilde{\nabla}_X(fY)+F^{k,l}(X,fY)\\
\nonumber
&=[X,f]\wedge Y+f\widetilde{\nabla}_X(Y)+fF^{k,l}(X,Y)\\
\nonumber
&=[X,f]\wedge Y+f\nabla_X Y.
\end{align}
This proves that $\nabla$ is a higher connection.   Furthermore, by the argument given in the first part of the proof, the pair $(\widetilde{\nabla},\{F^{k,l}\})$ can be recovered from $\nabla$.  This proves the bijection between the space of higher connections and all pairs of the form $(\widetilde{\nabla},\{F^{k,l}\})$ where $\widetilde{\nabla}$ is an affine connection on $TM$ and $F^{k,l}\in \Gamma(E^{k,l})$ for all $k,l>0$ with $k+l-1\le n$ and $F^{1,1}\equiv 0$.
\end{proof}
\begin{remark}
Let $\nabla$ be a higher connection and let $(\widetilde{\nabla},\{F^{k,l}\})$ be the unique pair associated to $\nabla$ by Theorem \ref{HigherConnectionData}.  For convenience, we set $F^{k,l}\equiv 0$ whenever $kl=0$ or $k+l-1>\dim M$.  With these definitions, Theorem \ref{HigherConnectionData} implies 
\begin{equation}
\nabla_X Y=\widetilde{\nabla}_X Y+F^{k,l}(X,Y),~\forall~ k,l\ge 0.
\end{equation}
\end{remark}
\begin{theorem}
\label{InducedResultA}
Let $\nabla$ be a higher connection and let $(\widetilde{\nabla},\{F^{k,l}\})$ be the unique pair associated to $\nabla$ by Theorem \ref{HigherConnectionData}.  Then $\nabla$ is an induced higher connection (i.e., one that satisfies (\ref{InducedHC1}) and (\ref{InducedHC2})) iff $F^{k,l}\equiv 0$, $\forall~k,l$. 
\end{theorem}

\begin{proof}
Extend $\widetilde{\nabla}$ to a higher connection via (\ref{InducedHC1}) and (\ref{InducedHC2}).

$(\Rightarrow)$. Let $\nabla$ be an induced higher connection and suppose that not all $F^{i,j}=0$.  Since $F^{1,1}\equiv 0$ by Theorem \ref{HigherConnectionData}, there are three possible cases:
\begin{itemize}
\item[1.] $F^{k,1}\neq 0$ for some $k>1$
\item[2.] $F^{1,l}\neq 0$ for some $l>1$
\item[3.] $F^{k,l}=0$ whenever $k=1$ or $l=1$, but there exists some $k,l>1$ such that $F^{k,l}\neq 0$
\end{itemize}
For case 1, let $k:=\min\{i~|~F^{i,1}\neq 0\}$.  Since $F^{k,1}\neq 0$, there exists $X,Y\in A^{1}(M)$ and $X'\in A^{k-1}(M)$ such that $F^{k,1}(X\wedge X', Y)\neq 0$.  Then
\begin{align}
\nonumber
\nabla_{X\wedge X'}Y&=\widetilde{\nabla}_{X\wedge X'}Y+F^{k,1}(X'',Y)\\
\nonumber
&=X\wedge \widetilde{\nabla}_{X'}Y+(-1)^{k-1}X'\wedge \widetilde{\nabla}_XY+F^{k,1}(X\wedge X',Y)\\
\label{Case1AFkl}
&=X\wedge \nabla_{X'}Y+(-1)^{k-1}X'\wedge \nabla_XY+F^{k,1}(X\wedge X',Y),
\end{align}
where the third equality follows from the fact that $F^{i,1}=0$ for $i<k$.  On the other hand, since $\nabla$ is an induced higher connection, we also have
\begin{equation}
\label{Case1BFkl}
\nabla_{X\wedge X'}Y=X\wedge \nabla_{X'}Y+(-1)^{k-1}X'\wedge \nabla_XY.
\end{equation}
Comparing (\ref{Case1AFkl}) and (\ref{Case1BFkl}) shows that $F^{k,1}(X\wedge X',Y)=0$, which is a contradiction.  Hence, $F^{k,1}=0$ for all $k$.  

For case 2, let $l:=\min\{j~|~F^{1,j}\neq 0\}$.  Since $F^{1,l}\neq 0$, there exists $X,Y\in A^1(M)$ and $Y'\in A^{l-1}(M)$ such that $F^{1,l}(X,Y\wedge Y')\neq 0$.  Then
\begin{align}
\nonumber
\nabla_X (Y\wedge Y')&=\widetilde{\nabla}_X(Y\wedge Y')+F^{1,l}(X,Y\wedge Y')\\
\nonumber
&=(\widetilde{\nabla}_X Y)\wedge Y'+Y\wedge \widetilde{\nabla}_X Y'+F^{1,l}(X,Y\wedge Y')\\
\label{Case2AFkl}
&=(\nabla_X Y)\wedge Y'+Y\wedge \nabla_X Y'+F^{1,l}(X,Y\wedge Y'),
\end{align}
where the the third equaltiy follows from the fact that $F^{1,j}=0$ for $j<l$.  Since $\nabla$ is induced, we also have
\begin{equation}
\label{Case2BFkl}
\nabla_X (Y\wedge Y')=(\nabla_X Y)\wedge Y'+Y\wedge \nabla_X Y'.
\end{equation}
Comparing (\ref{Case2AFkl}) with (\ref{Case2BFkl}) shows that $F^{1,l}(X,Y\wedge Y')=0$, which is a contradiction.  Hence, $F^{1,l}=0$ for all $l$.

For case 3, there exists $a,b>1$ such that $F^{a,b}\neq 0$.  Let $k:=\min\{i~|~F^{i,b}\neq 0\}$.  (Note that by hypothesis, we have $k>1$.).  Since $F^{k,b}\neq 0$, there exists $X\in A^1(M)$, $X'\in A^{k-1}(M)$, and $Y\in A^b(M)$ such that $F^{k.b}(X\wedge X',Y)\neq 0$.  At this point, the rest of the proof proceeds exactly as in case 1 leading to the contradiction that $F^{k,b}(X\wedge X',Y)=0$.  From this, we conclude that $F^{i,j}=0$ for all $i,j$.  

$(\Leftarrow)$ If $F^{k,l}=0$ for all $k,l$, then Theorem \ref{HigherConnectionData} gives $\nabla=\widetilde{\nabla}$ and the latter is an induced higher connection.
\end{proof}

\begin{definition}
Let $\nabla$ be a higher connection and let $(\widetilde{\nabla},\{F^{k,l}\})$ be the unique pair associated to $\nabla$ by Theorem \ref{HigherConnectionData}.  The tensor fields $F^{k,l}$ are called the \textit{twist fields} of the connection.
\end{definition}

\noindent Theorem \ref{InducedResultA} shows that any induced higher connection (i.e., one satisfying \textit{both} (\ref{InducedHC1}) and (\ref{InducedHC2})) must have vanishing twist fields.  This fact motivates the following two definitions:  
\begin{definition}
\label{UpperInduced}
A higher connection $\nabla$ is called \textit{upper induced} if it satisfies 
\begin{equation}
\nonumber
\nabla_X (Y\wedge Z) = (\nabla_XY)\wedge Z+(-1)^{(k-1)l}Y\wedge \nabla_X Z
\end{equation}
for $X\in A^k(M)$, $Y\in A^l(M)$ and $Z\in \mc{A}(M)$.
\end{definition}
\begin{definition}
\label{LowerInduced}
A higher connection $\nabla$ is called \textit{lower induced} if it satisfies
\begin{equation}
\nonumber
\nabla_{X\wedge Y} Z=X\wedge \nabla_Y Z+ (-1)^{kl} Y\wedge \nabla_X Z,
\end{equation}
for $X\in A^k(M)$, $Y\in A^l(M)$, and $Z\in \mc{A}(M)$.
\end{definition}
\noindent The notion of upper and lower induced puts the following restrictions on the twist fields:
\begin{proposition}
\label{UpperInducedProp}
Let $\nabla$ be a higher connection and let $(\widetilde{\nabla},\{F^{k,l}\})$ be the unique pair associated to $\nabla$ by Theorem \ref{HigherConnectionData}.  $\nabla$ is upper induced iff 
\begin{equation}
\label{Special0}
F^{k,l+m}(X,Y\wedge Z)=F^{k,l}(X,Y)\wedge Z+(-1)^{(k-1)l}Y\wedge F^{k,m}(X,Z)
\end{equation}
for $X\in A^k(X)$, $Y\in A^l(M)$, and $Z\in {A}^m(M)$, with $k,l,m>0$ and $k+l+m-1\le n:=\dim M$.  In particular, if $\nabla$ is upper induced, then the twist fields are completely determined by the set $\{F^{k,1}\}_{k=1}^n$.  For a decomposable $l$-vector field $Y=Y_1\wedge Y_2\wedge \cdots \wedge Y_l$, $F^{k,l}$ is given by
\begin{equation}
\label{Special4}
F^{k,l}(X,Y)=\sum_{j=1}^l(-1)^{j-1}F^{k,1}(X,Y_j)\wedge Y_1\wedge \cdots \wedge \widehat{Y_j}\wedge \cdots \wedge Y_l.
\end{equation}
In particular, $F^{1,l}\equiv 0$ for all $l$.  
\end{proposition}
\begin{proof}
Let $X\in A^k(X)$, $Y\in A^l(M)$, and $Z\in {A}^m(M)$, with $k+l+m-1\le n:=\dim M$.  By Theorem \ref{HigherConnectionData}, we have
\begin{align}
\nonumber
\nabla_X(Y\wedge Z)&=\widetilde{\nabla}_X(Y\wedge Z)+F^{k,l+m}(X,Y\wedge Z)\\
\label{Special1}
&=(\widetilde{\nabla}_XY)\wedge Z+(-1)^{(k-1)l}Y\wedge \widetilde{\nabla}_XZ+F^{k,l+m}(X,Y\wedge Z),
\end{align}
where the last equality follows from the fact that $\widetilde{\nabla}$ is induced.  Now suppose that $\nabla$ is upper induced.  Then 
\begin{align}
\nonumber
\nabla_X(Y\wedge Z)&=(\nabla_X Y)\wedge Z+(-1)^{(k-1)l}Y\wedge\nabla_X Z\\
\label{Special2}
&=(\widetilde{\nabla}_X Y)\wedge Z+F^{k,l}(X,Y)\wedge Z\\
\nonumber
&+(-1)^{(k-1)l}Y\wedge \widetilde{\nabla}_XZ+(-1)^{(k-1)l}Y\wedge F^{k,m}(X,Z).
\end{align}
Comparing (\ref{Special1}) and (\ref{Special2}) gives
\begin{equation}
\label{Special3}
F^{k,l+m}(X,Y\wedge Z)=F^{k,l}(X,Y)\wedge Z+(-1)^{(k-1)l}Y\wedge F^{k,m}(X,Z).
\end{equation}
On the other hand, if the twist fields satisfy (\ref{Special3}), then it follows that $\nabla$ is upper induced; this can be easily seen by substituting (\ref{Special3}) into (\ref{Special1}), rearranging the terms, and applying Theorem \ref{HigherConnectionData}.  For the last part, note that (\ref{Special4}) follows from  (\ref{Special3}) by a straightforward calculation, and $F^{1,l}\equiv 0$ since $F^{1,1}\equiv 0$ by Theorem \ref{HigherConnectionData}.  
\end{proof}

\begin{proposition}
\label{LowerInducedProp}
Let $\nabla$ be a higher connection and let $(\widetilde{\nabla},\{F^{k,l}\})$ be the unique pair associated to $\nabla$ by Theorem \ref{HigherConnectionData}.  $\nabla$ is lower induced iff
\begin{equation}
\nonumber
F^{k+l,m}(X\wedge Y,Z)=X\wedge F^{l,m}(Y,Z)+(-1)^{kl}Y\wedge F^{k,m}(X,Z),
\end{equation}
for $X\in A^k(M)$, $Y\in A^l(M)$, and $Z\in A^m(M)$, with $k+l+m-1\le n:=\dim M$.  In particular, if $\nabla$ is lower induced, then the twist fields are completely determined by the set $\{F^{1,l}\}_{l=1}^n$.  For a decomposable $k$-vector field $X=X_1\wedge \cdots \wedge X_k$, $F^{k,l}$ is given by
\begin{equation}
\label{Lower0}
F^{k,l}(X,Y)=\sum_{j=1}^k (-1)^{k-j} X_1\wedge \cdots \wedge \widehat{X_j}\wedge \cdots \wedge X_k\wedge F^{1,l}(X_j,Y).
\end{equation}
In particular, $F^{k,1}\equiv 0$ for all $k$.
\end{proposition}
\begin{proof}
Let $X\in A^k(M)$, $Y\in A^l(M)$, and $Z\in A^m(M)$, with $k+l+m-1\le n:=\dim M$.  By Theorem \ref{HigherConnectionData}, we have
\begin{align}
\nonumber
\nabla_{X\wedge Y}Z&=\widetilde{\nabla}_{X\wedge Y}Z+F^{k+l,m}(X\wedge Y,Z)\\
\label{Lower1}
&=X\wedge \widetilde{\nabla}_Y Z+(-1)^{kl} Y\wedge\widetilde{\nabla}_X Z+F^{k+l,m}(X\wedge Y,Z),
\end{align}
where the last equality follows from the fact that $\widetilde{\nabla}$ is induced.  Now suppose that $\nabla$ is lower induced.  Then
\begin{align}
\nonumber
\nabla_{X\wedge Y} Z&=X\wedge \nabla_Y Z+(-1)^{kl} Y\wedge \nabla_X Z\\
\nonumber
&=X\wedge \widetilde{\nabla}_Y Z+X\wedge F^{l,m}(Y,Z)\\
\label{Lower2}
&+(-1)^{kl}Y\wedge \widetilde{\nabla}_X Z+(-1)^{kl}Y\wedge F^{k,m}(X,Z).
\end{align}
Comparing (\ref{Lower1}) and (\ref{Lower2}) gives 
\begin{equation}
\label{Lower3}
F^{k+l,m}(X\wedge Y,Z)=X\wedge F^{l,m}(Y,Z)+(-1)^{kl}Y\wedge F^{k,m}(X,Z).
\end{equation}

On the other hand, if the twist fields satisfy (\ref{Lower3}), then substitution of (\ref{Lower3}) into (\ref{Lower1}) shows that $\nabla$ is lower induced.   

For the last part, note that (\ref{Lower0}) follows from  (\ref{Lower3}) by a straightforward calculation and $F^{k,1}\equiv 0$ for all $k$ since $F^{1,1}\equiv 0$ by Theorem \ref{HigherConnectionData}.   
\end{proof}

\noindent We now introduce the notion of \textit{higher torsion} for higher connections:
\begin{definition}
\label{HigherTorsion}
Let $\nabla$ be a higher connection.  The \textit{higher torsion} associated to $\nabla$ is defined by
\begin{equation}
T(X,Y):=\nabla_XY-(-1)^{(k-1)(l-1)}\nabla_Y X-[X,Y]
\end{equation} 
for $X\in A^k(M)$, $Y\in A^l(M)$.  $\nabla$ is \textit{torsion-free} if $T\equiv 0$.  
\end{definition}

\begin{proposition}
\label{HigherTorsionProp}
Let $\nabla$ be a higher connection with higher torsion $T$.  For $X\in A^k(M)$, $Y\in A^l(M)$, and $f\in C^\infty(M)$, $T$ satisfies
\begin{itemize}
\item[(i)] $T(X,Y)=-(-1)^{(k-1)(l-1)}T(Y,X)$
\item[(ii)] $T(fX,Y)=fT(X,Y)$
\end{itemize}
\end{proposition}
\begin{proof}
\noindent Let $s=(-1)^{(k-1)(l-1)}$.  For (i), we have
\begin{align}
\nonumber
T(X,Y)&:=\nabla_XY-s\nabla_Y X-[X,Y]\\
\nonumber
&=-s(\nabla_Y X - s\nabla_X Y + s[X,Y])\\
\nonumber
&=-s(\nabla_Y X - s\nabla_X Y + s(-s[Y,X])\\
\nonumber
&=-s(\nabla_Y X - s\nabla_X Y - [Y,X])\\
\nonumber
&=-sT(Y,X).
\end{align}

For (ii), we have
\begin{align}
\nonumber
T(fX,Y)&=\nabla_{fX}Y-s\nabla_Y(fX)-[fX,Y]\\
\nonumber
&=f\nabla_X Y-s([Y,f]\wedge X+f\nabla_Y X)-(f[X,Y]-X\wedge i_fY)\\
\nonumber
&=f\nabla_X Y-s((-1)^{l-1}i_f Y\wedge X+f\nabla_Y X)-(f[X,Y]-X\wedge i_fY)\\
\nonumber
&=f\nabla_X Y-(-1)^{k(l-1)}i_f Y\wedge X-sf\nabla_Y X-f[X,Y]+ (-1)^{k(l-1)} i_fY\wedge X\\
\nonumber
&=f\nabla_X Y-sf\nabla_Y X-f[X,Y]\\
\nonumber
&=fT(X,Y).
\end{align}
where we used Proposition \ref{InteriorProductProp}-(vi) and (v) in the second and third equalities respectively.
\end{proof}

\noindent The next two results provide a characterization of torsion-free higher connections.
\begin{proposition}
\label{VanishingTorsionProp}
Let $\widetilde{\nabla}$ be an affine connection on $TM$.  Then the following statements are equivalent:  
\begin{itemize}
\item[(i)] $\widetilde{\nabla}$ is torsion-free as an affine connection on $TM$.
\item[(ii)] $\widetilde{\nabla}$ is torsion-free as an induced higher connection.
\end{itemize}
\end{proposition}
\begin{proof}
$(i) \Leftarrow (ii)$. Immediate.

$(i) \Rightarrow (ii)$. Suppose that $\widetilde{\nabla}$ is torsion-free as an affine connection on $TM$.  Extend $\widetilde{\nabla}$ to a higher connection via (\ref{InducedHC1}) and (\ref{InducedHC2}) and let $T$ denote its higher torsion.  To prove that $T\equiv 0$, it suffices to show that $T(X,Y)\equiv 0$ for the case when $X$ and $Y$ are decomposable $k$ and $l$-vector fields respectively.  So, let $X=X_1\wedge \cdots \wedge X_k$ and $Y=Y_1\wedge \cdots \wedge Y_l$.  To simplify things, write
\begin{align}
\nonumber
X[i] &= X_1\wedge\cdots \wedge \widehat{X}_i\wedge \cdots \wedge X_k\\
Y[j] &= Y_1\wedge\cdots \wedge \widehat{Y}_j\wedge \cdots \wedge Y_l,
\end{align}
where $\widehat{X}_i$ and $\widehat{Y}_j$ denotes omission as usual.  Using (\ref{ExplicitForm1}) and (\ref{ExplicitForm2}) from the proof of Proposition \ref{InducedHC1}, we have
\begin{align}
\nonumber
\widetilde{\nabla}_XY&=\sum_{i=1}^k(-1)^{k-i}X[i]\wedge \widetilde{\nabla}_{X_i} Y\\
\nonumber
&=\sum_{i=1}^{k}\sum_{j=1}^l (-1)^{k-i}(-1)^{j-1}X[i]\wedge \widetilde{\nabla}_{X_i} Y_j\wedge Y[j]\\
\nonumber
&=\sum_{i=1}^{k}\sum_{j=1}^l (-1)^{i+j}\widetilde{\nabla}_{X_i} Y_j\wedge X[i]\wedge Y[j].
\end{align}
Likewise,
\begin{equation}
\nonumber
\widetilde{\nabla}_YX=\sum_{i=1}^{k}\sum_{j=1}^l (-1)^{i+j}\widetilde{\nabla}_{Y_j} X_i\wedge Y[j]\wedge X[i].
\end{equation}
Recall that for $X$, $Y$ decomposable, the Schouten-Nijenhuis bracket is given by
\begin{equation}
\nonumber
[X,Y]=\sum_{i=1}^k\sum_{j=1}^l(-1)^{i+j}[X_i,Y_j]\wedge X[i]\wedge Y[j].
\end{equation}
Putting everything together gives
\begin{align}
\nonumber
T(X,Y)&=\widetilde{\nabla}_XY-(-1)^{(k-1)(l-1)}\widetilde{\nabla}_YX-[X,Y]\\
\nonumber
&=\sum_{i=1}^{k}\sum_{j=1}^l (-1)^{i+j}(\widetilde{\nabla}_{X_i} Y_j-\widetilde{\nabla}_{Y_j} X_i-[X_i,Y_j])\wedge X[i]\wedge Y[j]\\
\nonumber
&=0.
\end{align}
This completes the proof.
\end{proof}
\begin{theorem}
\label{TorsionThmB}
Let $\nabla$ be a higher connection and let $(\widetilde{\nabla},\{F^{k,l}\})$ be the unique pair associated to $\nabla$ by Theorem \ref{HigherConnectionData}.  Then $\nabla$ is torsion-free iff 
\begin{itemize}
\item[(i)] $\widetilde{\nabla}$ is torsion-free, and 
\item[(ii)] $F^{k,l}(X,Y)=(-1)^{(k-1)(l-1)}F^{l,k}(Y,X)$ for all $X\in A^k(M)$, $Y\in A^l(M)$.  
\end{itemize}
\end{theorem}
\begin{proof}
Let $X\in A^k(M)$ and $Y\in A^l(M)$ and let $T$ denote the higher torsion associated with $\nabla$. Note that if $k=0$ or $l=0$, then $T(X,Y)=0$.  Consequently, assume that $k,l>0$.  Then
\begin{align}
\nonumber
T(X,Y)&=\nabla_XY-(-1)^{(k-1)(l-1)}\nabla_Y X-[X,Y]\\
\nonumber
&=\widetilde{\nabla}_XY+F^{k,l}(X,Y)-(-1)^{(k-1)(l-1)}\widetilde{\nabla}_Y X-(-1)^{(k-1)(l-1)}F^{l,k}(Y,X)\\
\nonumber
&-[X,Y]\\
\label{TorsionThmB1}
&=\widetilde{T}(X,Y)+F^{k,l}(X,Y)-(-1)^{(k-1)(l-1)}F^{l,k}(Y,X),
\end{align}
where $\widetilde{T}$ denotes the higher torsion associated to $\widetilde{\nabla}$ (as an induced higher connection).  

Now suppose that $T\equiv 0$.  Since $F^{1,1}\equiv 0$, we have $\widetilde{T}(X,Y)=T(X,Y)=0$ for all $X,Y\in A^1(M)$.  Consequently, $\widetilde{\nabla}$ is torsion free as an affine connection on $TM$.   Proposition \ref{VanishingTorsionProp} then implies that $\widetilde{T}\equiv 0$.  This fact along with (\ref{TorsionThmB1}) implies that 
\begin{equation}
\nonumber
F^{k,l}(X,Y)=(-1)^{(k-1)(l-1)}F^{l,k}(Y,X).
\end{equation}
The converse follows immediately from (\ref{TorsionThmB1}).  This completes the proof.  
\end{proof}
\begin{corollary}
\label{TorsionGkl}
Let $\nabla$ be a higher connection and let $(\widetilde{\nabla},\{F^{k,l}\})$ be the unique pair associated to $\nabla$ by Theorem \ref{HigherConnectionData}.  Let 
\begin{equation}
\nonumber
G^{k,l}(X,Y):=F^{k,l}(X,Y)+(-1)^{(k-1)(l-1)}F^{l,k}(Y,X),
\end{equation}
for $X\in A^k(M)$, $Y\in A^l(M)$.  Let $\nabla'$ be the unique higher connection associated to $(\widetilde{\nabla},\{G^{k,l}\})$ by Theorem \ref{HigherConnectionData}.  If $\widetilde{\nabla}$ is torsion-free as an affine connection on $TM$, then $\nabla'$ has vanishing higher torsion.  
\end{corollary}
\begin{proof}
This follows immediately from Theorem \ref{TorsionThmB} by noting that 
\begin{equation}
\nonumber
G^{k,l}(X,Y)=(-1)^{(k-1)(l-1)} G^{l,k}(Y,X)
\end{equation}
for $X\in A^k(M)$, $Y\in A^l(M)$.
\end{proof}

\begin{corollary}
\label{UpperLowerTorsionCor}
Let $\nabla$ be an upper-induced (or lower-induced) higher connection.  If $\nabla$ is torsion-free, then $\nabla$ is induced by a torsion-free affine connection on $TM$.  
\end{corollary}
\begin{proof}
Suppose that $\nabla$ is a torsion-free, upper-induced higher connection, and let $(\widetilde{\nabla},F^{k,l})$ be the unique pair associated to $\nabla$ by Theorem \ref{HigherConnectionData}.  By Theorem \ref{TorsionThmB},  we have
\begin{equation}
\label{ULTorsionCor1}
F^{k,l}(X,Y)=(-1)^{(k-1)(l-1)}F^{l,k}(Y,X)
\end{equation}
for all $X\in A^k(M)$ and $Y\in A^l(M)$.  In particular, $F^{k,1}(X,Y)=F^{1,k}(Y,X)$ for all $X\in A^k(M)$ and $Y\in A^1(M)$.  By Proposition \ref{UpperInducedProp}, $F^{1,k}\equiv 0$ for all $k$.   Equation (\ref{ULTorsionCor1}) (with $l=1$) then implies that $F^{k,1}\equiv 0$ for all $k$.  Applying Proposition \ref{UpperInducedProp} once more, we have $F^{k,l}\equiv 0$ for all $k,l$.  This shows that $\nabla=\widetilde{\nabla}$ and $\widetilde{\nabla}$ is torsion-free by Theorem \ref{TorsionThmB}.  The proof for the case when $\nabla$ is torsion-free and lower induced is similar.  
\end{proof}

\noindent We conclude this section with the following observation:
\begin{proposition}
\label{FlatnessProp}
Let $\nabla$ be a higher connection and let $(\widetilde{\nabla},\{F^{k,l}\})$ be the unique pair associated with $\nabla$ by Theorem \ref{HigherConnectionData}.  Then the following statements are equivalent:
\begin{itemize}
\item[(i)] $\widetilde{\nabla}$ is flat as an affine connection on $TM$.
\item[(ii)] For any $p\in M$ and any $y\in \wedge^l T_pM$, $l>0$, there exists a neighborhood $U\subset M$ of $p$ and a $Y\in A^l(U)$ such that $Y_p=y$ and $\nabla_XY=F^{k,l}(X,Y)$ for all $X\in A^k(U)$, $k>0$. 
\end{itemize}
\end{proposition}
\begin{proof}
(i)$\Rightarrow$(ii).  Let $\widetilde{\nabla}$ be flat as an affine connection on $TM$ and let $p\in M$.  To start, consider the case when $y\in \wedge^l T_pM$ is a decomposable $l$-vector.  Then $y=y_1\wedge \cdots \wedge y_l$ for some $y_j\in T_pM$, $j=1,\dots, l$.  Since $\widetilde{\nabla}$ is flat, there exists a neighborhood $U_i$ of $p$ and a $Y_j\in A^1(U_j)$ such that $(Y_j)_p=y_j$ and $\widetilde{\nabla}_X Y_j=0$ on $U_j$ for all $X\in A^1(U_j)$, $j=1,\dots, l$.  Let $U=U_1\cap \cdots \cap U_l$, $Y=Y_1\wedge \cdots \wedge Y_l$, and let $X=X_1\wedge \cdots \wedge X_k$ be any decomposable $k$-vector field on $U$, $k>0$.  Since $\widetilde{\nabla}$ is induced, we have 
\begin{align}
\label{FlatnessProp1}
\widetilde{\nabla}_X Y&=\sum_{i=1}^k \sum_{j=1}^l (-1)^{k-i} (-1)^{j-1}X[i]\wedge \widetilde{\nabla}_{X_i}Y_j \wedge Y[j]=0,
\end{align}
where $X[i]=X_1\wedge \cdots \wedge \widehat{X_i}\wedge \cdots \wedge X_k$ and $Y[j]=Y_1\wedge \cdots \wedge \widehat{Y_j}\wedge \cdots \wedge Y_l$.  By linearity, (\ref{FlatnessProp1}) holds for all $X\in A^k(U)$.  

Now let $y$ be any $l$-vector at $p$.   Then $y$ is a sum of decomposable $l$-vectors: $y=y_{(1)}+\cdots  + y_{(t)}$, where $y_{(a)}$ is a decomposable $l$-vector for $a=1,\dots, t$.  By the above argument, there exists a neighborhood $U_{(a)}$ of $p$ and a decomposable $l$-vector field $Y_{(a)}$ on $U_{(a)}$ such that $(Y_{(a)})_p=y_{(a)}$ and $\widetilde{\nabla}_X Y_{(a)}=0$ for all $X\in A^k(U_{(a)})$, $a=1,\dots, t$.  Now set $U=U_{(1)}\cap \cdots \cap U_{(a)}$ and $Y=Y_{(1)}+\cdots +Y_{(t)}$.  Then $Y_p=y$ and $\widetilde{\nabla}_XY=0$ for all $X\in A^k(U)$, $k>0$.  This fact along with Theorem \ref{HigherConnectionData} gives 
\begin{equation}
\nonumber
\nabla_X Y=\widetilde{\nabla}_X Y +F^{k,l}(X,Y)=F^{k,l}(X,Y),\hspace*{0.1in}\forall X\in A^k(U),~k>0.
\end{equation}

(i)$\Leftarrow$(ii).  Let $\nabla$ be a higher connection satisfying (ii).  From Theorem \ref{HigherConnectionData}, $\nabla$ and $\widetilde{\nabla}$ (as an induced higher connection), coincide on $A^1(M)\times A^1(M)$ (i.e., $F^{1,1}\equiv 0$).  Consequently, for the special case of $k=l=1$, (ii) implies that for any $p\in M$ and any $y\in T_pM$, there exists a neighborhood $U$ of $p$ and a $Y\in A^1(U)$ such that $Y_p=y$ and $\widetilde{\nabla}_X Y=0$ on $U$ for all $X\in A^1(U)$.  In other words, $\widetilde{\nabla}|_{A^1(M)\times A^1(M)}$ is a flat affine connection.  This completes the proof.
\end{proof}
\begin{corollary}
Let $\widetilde{\nabla}$ be a flat affine connection on $TM$ and extend $\widetilde{\nabla}$ to an induced higher connection.  Then for any $p\in M$ and any $y\in \wedge^l T_pM$, $l>0$, there exists a neighborhood $U\subset M$ of $p$ and a $Y\in A^l(U)$ such that $Y_p=y$ and  $\widetilde{\nabla}_XY=0$ for all $X\in A^k(U)$, $k>0$.
\end{corollary}

\section{Extension to Differential Forms}
Let $\nabla$ be a higher connection,  $\omega \in \Omega^l(M)$, and  $X\in A^k(M)$.  For $k>0$ and $l-k+1\ge 0$, we define $\nabla_X\omega\in \Omega^{l-k+1}(M)$ via
\begin{align}
\label{CovDerDiffForm}
\nabla_X\omega(Y):=(-1)^{(k-1)(l-1)}L_X i_Y\omega-\omega(\nabla_X Y)
\end{align}
for all $Y\in A^{l-k+1}(M)$.  (For $\eta\in \Omega^0(M)$ and $f\in A^0(M):=C^\infty(M)$, we understand $\eta(f)$ to mean $f\eta$.)  Lastly, for $k=0$ or $l-k+1<0$, we set $\nabla_X\omega:=0$.  We will now verify that $\nabla_X\omega$ is indeed an $(l-k+1)$-form.  To do this, we need the following lemmas:
\begin{lemma}
\label{InteriorProduct2}
Let $X=X_1\wedge \cdots \wedge X_k$ be a decomposable $k$-vector field and let $\alpha\in \Omega^1(M)$ and $\omega\in \Omega^\bullet(M)$.  Then
\begin{equation}
\nonumber
i_X(\alpha\wedge \omega)=\sum_{j=1}^k(-1)^{j+1}\alpha(X_j)i_{X_1\wedge \cdots \wedge \wh{X}_j\wedge \cdots \wedge X_k}\omega+(-1)^k\alpha\wedge i_X\omega.
\end{equation}
\end{lemma}
\begin{proof}
We prove Lemma \ref{InteriorProduct2} by induction on $k$.  For $k=1$, $X$ is a vector field and (\ref{InteriorProductDeg1}) gives   
\begin{equation}
i_X(\alpha\wedge \omega)=\alpha(X)\omega-\alpha\wedge i_X\omega.
\end{equation}
Hence,  Lemma \ref{InteriorProduct2} holds for $k=1$.  Now suppose that Lemma \ref{InteriorProduct2} holds for $X=X_1\wedge \cdots \wedge X_k$.  Let $X_{k+1}\in A^1(M)$.   By Proposition \ref{InteriorProduct1}, we have
\begin{align}
\nonumber
i_{X\wedge X_{k+1}}(\alpha\wedge \omega) &= i_{X_{k+1}} i_X(\alpha\wedge \omega)\\
\nonumber
&=i_{X_{k+1}}\left(\sum_{j=1}^k(-1)^{j+1}\alpha(X_j)i_{X_1\wedge \cdots \wedge \wh{X}_j\wedge \cdots \wedge X_k}\omega+(-1)^k\alpha\wedge i_X\omega\right)\\
\nonumber
&=\sum_{j=1}^k(-1)^{j+1}\alpha(X_j)i_{X_{k+1}}i_{X_1\wedge \cdots \wedge \wh{X}_j\wedge \cdots \wedge X_k}\omega+(-1)^ki_{X_{k+1}}(\alpha\wedge i_X\omega)\\
\nonumber
&=\sum_{j=1}^k(-1)^{j+1}\alpha(X_j)i_{X_1\wedge \cdots \wedge \wh{X}_j\wedge \cdots \wedge X_{k+1}}\omega\\
\nonumber
&+(-1)^k(\alpha(X_{k+1}) i_X\omega-\alpha\wedge i_{X_{k+1}}i_X\omega)\\
\nonumber
&=\sum_{j=1}^k(-1)^{j+1}\alpha(X_j)i_{X_1\wedge \cdots \wedge \wh{X}_j\wedge \cdots \wedge X_{k+1}}\omega+(-1)^{k+2}\alpha(X_{k+1}) i_X\omega\\
\nonumber
&+(-1)^{k+1}\alpha\wedge i_{X\wedge X_{k+1}}\omega\\
\nonumber
&=\sum_{j=1}^{k+1}(-1)^{j+1}\alpha(X_j)i_{X_1\wedge \cdots \wedge \wh{X}_j\wedge \cdots \wedge X_{k+1}}+(-1)^{k+1}\alpha\wedge i_{X\wedge X_{k+1}}\omega,
\end{align}
where we have used the induction hypothesis in the second equality, and Proposition \ref{InteriorProduct1} and equation (\ref{InteriorProductDeg1}) in the fourth equality.  This completes the proof.
\end{proof}

\begin{lemma}
\label{SchoutenLemma}
Let $X=X_1\wedge \cdots \wedge X_k$ be a decomposable $k$-vector field and let $f\in C^\infty(M)$.  Then
\begin{equation}
\nonumber
[X,f]=\sum_{j=1}^k(-1)^{k-j} (X_jf) X_1\wedge \cdots \wedge\wh{X}_j\wedge \cdots \wedge X_k
\end{equation}
\end{lemma}
\begin{proof}
We prove this by induction on $k$.  For $k=1$, we have $X=X_1$ and $[X_1,f]:=X_1f$.  Hence, the lemma holds for $k=1$.  Suppose that the lemma holds for $X=X_1\wedge \cdots \wedge X_k$.  Let $X_{k+1}\in A^1(M)$.  Then
\begin{align}
\nonumber
[X\wedge X_{k+1},f]&=-(-1)^k[f,X\wedge X_{k+1}]\\
\nonumber
&=-(-1)^k([f,X]\wedge X_{k+1}+(-1)^k X\wedge [f,X_{k+1}])\\
\nonumber
&=-(-1)^k(-(-1)^{k-1}[X,f]\wedge X_{k+1}+(-1)^{k+1} X\wedge [X_{k+1},f])\\
\nonumber
&=-(-1)^k((-1)^{k}[X,f]\wedge X_{k+1}+(-1)^{k+1}  (X_{k+1}f)X\\
\nonumber
&=-[X,f]\wedge X_{k+1}+ (X_{k+1}f)X\\
\nonumber
&=\sum_{j=1}^k(-1)^{k+1-j} (X_jf) X_1\wedge \cdots\wedge \wh{X}_j\wedge \cdots \wedge X_k\wedge X_{k+1}+ (X_{k+1}f)X\\
\nonumber
&=\sum_{j=1}^{k+1}(-1)^{k+1-j} (X_jf) X_1\wedge \cdots\wedge \wh{X}_j\wedge \cdots \wedge X_k\wedge X_{k+1},
\end{align}
where we have used the induction hypothesis in the sixth equality.   This completes the proof.
\end{proof}

\begin{lemma}
\label{InteriorProduct3}
Let $X\in A^k(M)$, $f\in C^\infty(M)$, and $\omega\in \Omega^\bullet(M)$.  Then 
\begin{equation}
\nonumber
L_Xf\omega = fL_X\omega+i_{[X,f]}\omega.
\end{equation}
\end{lemma}
\begin{proof}
Without loss of generality, let $X=X_1\wedge \cdots \wedge X_k$ be a decomposable $k$-vector field.  Then
\begin{align}
\nonumber
L_Xf\omega &=di_Xf\omega-(-1)^ki_Xd f\omega\\
\nonumber
&=d(fi_X\omega) - (-1)^ki_X(df\wedge \omega+fd\omega)\\
\nonumber
&=df\wedge i_X\omega+f d i_X\omega\\
\nonumber
&-(-1)^k\left(\sum_{j=1}^{k} (-1)^{j+1}(X_j f) i_{X_1\wedge\cdots \wedge\wh{X}_j \wedge \cdots \wedge X_k}\omega+(-1)^kdf\wedge i_X\omega+fi_Xd\omega\right)\\
\nonumber
&=fdi_X\omega+\sum_{j=1}^k(-1)^{k-j}(X_j f) i_{X_1\wedge\cdots \wedge\wh{X}_j \wedge \cdots \wedge X_k}\omega-(-1)^k f i_Xd\omega\\
\nonumber
&=fL_X\omega+\sum_{j=1}^k(-1)^{k-j}(X_j f) i_{X_1\wedge\cdots \wedge\wh{X}_j \wedge \cdots \wedge X_k}\omega\\
\nonumber
&=fL_X\omega+i_{[X,f]}\omega,
\end{align}
where  Lemma \ref{InteriorProduct2} is used in the third equality and Lemma \ref{SchoutenLemma} is used in the last equality.
\end{proof}
\begin{proposition}
\label{P925}
Let $\nabla$ be a higher connection.  Then $\nabla_X\omega=L_X\omega\in C^\infty(M)$ for $X\in A^k(M)$ and $\omega\in \Omega^{k-1}(M)$.
\end{proposition}
\begin{proof}
Let $f\in C^\infty(M)$.  Then
\begin{align}
\nonumber
f\nabla_X\omega&=\nabla_X\omega(f)\\
\nonumber
&:=(-1)^{(k-1)(k-2)}L_Xi_f\omega-\omega(\nabla_Xf)\\
\nonumber
&=L_Xf\omega-\omega([X,f])\\
\nonumber
&=fL_X\omega+i_{[X,f]}\omega-\omega([X,f])\\
\nonumber
&=fL_X\omega.
\end{align}
where Lemma \ref{InteriorProduct3} is used in the fourth equality.  This proves the proposition.
\end{proof}

\begin{theorem}
\label{CovDerDiffFormProp}
Let $\nabla$ be a higher connection, $\omega \in \Omega^l(M)$, and $X\in A^k(M)$ with $k> 0$. Then $\nabla_X \omega\in \Omega^{l-k+1}(M)$.
\end{theorem}
\begin{proof}
For $l-k+1\ge 0$, let $Y\in A^{l-k+1}(M)$.  It follows from (\ref{CovDerDiffForm}) that $\nabla_X\omega(Y)\in C^\infty(M)$.  To prove Theorem \ref{CovDerDiffFormProp}, it suffices to show
\begin{equation}
\label{CovDerDiffFormProp1}
\nabla_X\omega(fY)=f\nabla_X\omega(Y)
\end{equation}
for $f\in C^\infty(M)$.  We now verify (\ref{CovDerDiffFormProp1}):
\begin{align}
\nonumber
\nabla_X\omega(fY)&=(-1)^{(k-1)(l-1)}L_X i_{fY}\omega-\omega(\nabla_X fY)\\
\nonumber
&=(-1)^{(k-1)(l-1)}L_X fi_Y\omega-\omega([X,f]\wedge Y)-f\omega(\nabla_XY)\\
\nonumber
&=(-1)^{(k-1)(l-1)}fL_X i_Y\omega+(-1)^{(k-1)(l-1)}i_{[X,f]}i_Y\omega-\omega([X,f]\wedge Y)\\
\nonumber
&-f\omega(\nabla_XY)\\
\nonumber
&=f\nabla_X\omega(Y)+(-1)^{(k-1)(l-1)}i_{[X,f]}i_Y\omega-\omega([X,f]\wedge Y)\\
\nonumber
&=f\nabla_X\omega(Y)+(-1)^{(k-1)(l-1)}i_Y\omega([X,f])-\omega([X,f]\wedge Y)\\
\nonumber
&=f\nabla_X\omega(Y)+(-1)^{(k-1)(l-1)}\omega(Y\wedge [X,f])-\omega([X,f]\wedge Y)\\
\nonumber
&=f\nabla_X\omega(Y)+(-1)^{k(k-1)}\omega([X,f]\wedge Y)-\omega([X,f]\wedge Y)\\
\nonumber
&=f\nabla_X\omega(Y),
\end{align}
where Lemma \ref{InteriorProduct3} is used in the third equality.  This completes the proof.
\end{proof}
\noindent We now conclude this section with some of the properties of (\ref{CovDerDiffForm}).  Before doing so, we need a quick lemma: 
\begin{lemma}
\label{LieDerivativeLemma1}
Let $X\in A^k(M)$, $\eta\in \Omega^\bullet(M)$, and $f\in C^\infty(M)$.  Then 
\begin{equation}
\nonumber
L_{fX}\eta=df\wedge i_X\eta +fL_X\eta.
\end{equation}
\end{lemma}
\begin{proof}
From (\ref{LieDerivativeMVF}), we have
\begin{align}
\nonumber
L_{fX}\eta&=di_{fX}\eta-(-1)^ki_{fX}d\eta\\
\nonumber
&=dfi_X\eta-(-1)^kfi_Xd\eta\\
\nonumber
&=df\wedge i_X\eta+fdi_X\eta-(-1)^kfi_Xd\eta\\
\nonumber
&=df\wedge i_X\eta+fL_X\eta.
\end{align}
\end{proof}

\begin{theorem}
\label{CovDerDiffFormProp2}
Let $\nabla$ be a higher connection.  For $\omega \in \Omega^l(M)$, $X\in A^k(M)$ $(k>0)$, and $f\in C^\infty(M)$, $\nabla$ satisfies
\begin{itemize}
\item[(i)] $\nabla_{fX}\omega=f\nabla_X\omega$
\item[(ii)] $\nabla_Xf\omega=f\nabla_X\omega+i_{[X,f]} \omega$
\end{itemize}
If $\nabla$ is also upper induced, then
\begin{itemize}
\item[(iii)] $\nabla_X(i_W\omega)=(-1)^{j(k-1)}(i_W\nabla_X\omega+i_{\nabla_XW}\omega)$ for $W\in A^j(M)$. 
\end{itemize}
\end{theorem} 
\begin{proof}
Let $l-k+1\ge 0$ and let $Y\in A^{l-k+1}(M)$.  For (i), we have
\begin{align}
\nonumber
\nabla_{fX}\omega(Y)&=(-1)^{(k-1)(l-1)}L_{fX}i_Y\omega-\omega(\nabla_{fX}Y)\\
\nonumber
&=(-1)^{(k-1)(l-1)}df\wedge i_Xi_Y\omega+(-1)^{(k-1)(l-1)}fL_Xi_Y\omega-f\omega(\nabla_{X}Y)\\
\nonumber
&=(-1)^{(k-1)(l-1)}df\wedge i_{Y\wedge X}\omega+f\nabla_X\omega(Y)\\
\nonumber
&=f\nabla_X\omega(Y),
\end{align}
where Lemma \ref{LieDerivativeLemma1} was used in the second equality, Proposition  \ref{InteriorProduct1} was used in the third equality, and the last equality follows from the fact that $Y\wedge X\in A^{l+1}(M)$ and $\omega\in \Omega^l(M)$.

For (ii), we have
\begin{align}
\nonumber
\nabla_X f\omega (Y)&=(-1)^{(k-1)(l-1)}L_Xi_Y(f\omega)-(f\omega)(\nabla_XY)\\
\nonumber
&=(-1)^{(k-1)(l-1)}L_Xfi_Y\omega-(f\omega)(\nabla_XY)\\
\nonumber
&=(-1)^{(k-1)(l-1)}fL_Xi_Y\omega+(-1)^{(k-1)(l-1)}i_{[X,f]}i_Y\omega-(f\omega)(\nabla_XY)\\
\nonumber
&=f\nabla_X\omega(Y)+(-1)^{(k-1)(l-1)}i_{[X,f]}i_Y\omega\\
\nonumber
&=f\nabla_X\omega(Y)+(-1)^{k(k-1)}i_Yi_{[X,f]}\omega\\
\nonumber
&=f\nabla_X\omega(Y)+i_{[X,f]}\omega(Y),
\end{align}
where Lemma \ref{InteriorProduct3} is used in the third equality and Proposition \ref{InteriorProduct1} is used in the fifth equality.

For (iii), let $\nabla$ be an upper induced higher connection.  Note that for $j=0$, (iii) follows directly from (ii) of Theorem \ref{CovDerDiffFormProp2}.  Now assume that $j>0$.  Let $t=l-j-k+1\ge 0$ and let $Z\in A^t(M)$.  Then
\begin{align}
\nonumber
\nabla_X(i_W\omega)(Z)&=(-1)^{(l-j-1)(k-1)}L_Xi_Zi_W\omega-i_W\omega(\nabla_X Z)\\
\nonumber
&=(-1)^{(l-1)(k-1)}(-1)^{j(k-1)}L_Xi_{W\wedge Z}\omega-\omega(W\wedge\nabla_X Z)\\
\nonumber
&=(-1)^{(l-1)(k-1)}(-1)^{j(k-1)}L_Xi_{W\wedge Z}\omega-(-1)^{j(k-1)}\omega(\nabla_X(W\wedge Z))\\
\nonumber
&+(-1)^{j(k-1)}\omega(\nabla_X(W\wedge Z))-\omega(W\wedge\nabla_X Z)\\
\nonumber
&=(-1)^{j(k-1)}\nabla_X\omega(W\wedge Z)\\
\nonumber
&+(-1)^{j(k-1)}(\omega((\nabla_XW)\wedge Z)+(-1)^{j(k-1)}\omega(W\wedge \nabla_X Z))\\
\nonumber
&-\omega(W\wedge\nabla_X Z)\\
\nonumber
&=(-1)^{j(k-1)}\nabla_X\omega(W\wedge Z)+(-1)^{j(k-1)}\omega((\nabla_XW)\wedge Z)\\
\nonumber
&=(-1)^{j(k-1)}(i_W\nabla_X\omega(Z)+i_{\nabla_XW}\omega(Z)),
\end{align}
where Proposition \ref{InteriorProduct1} is used in the second equality, and (\ref{CovDerDiffForm}) and the fact that $\nabla$ is upper induced are used in the fourth equality.  This completes the proof. 
\end{proof}

\begin{theorem}
Let $\nabla$ be a torsion-free induced higher connection.  Then
\begin{equation}
\label{P924a}
\nabla_{X\wedge Y} \omega = (-1)^li_Y\nabla_X\omega+(-1)^{k(l-1)}i_X\nabla_Y\omega
\end{equation}
for $X\in A^k(M)$, $Y\in A^l(M)$, and $\omega\in \Omega^m(M)$.
\end{theorem}
\begin{proof}
For $m<k+l-1$, both sides of (\ref{P924a}) are zero.  In addition, for $k=0$ or $l=0$, (\ref{P924a}) follows from Theorem \ref{CovDerDiffFormProp2}-(i).  We now verify (\ref{P924a}) for $m\ge k+l-1$ with $k,l>0$. 

Let $Z\in A^t(M)$ where $t=m-k-l+1$.  Then 
\begin{align}
\label{P924b}
\nabla_{X\wedge Y}\omega(Z)=qL_{X\wedge Y} i_Z\omega-\omega(\nabla_{X\wedge Y} Z),
\end{align}
where $q:=(-1)^{(k+l-1)(m-1)}$.  Using (iv) and (ii) of Proposition \ref{LieDerivativeProp} and (i) of Proposition \ref{InteriorProduct1}, the first term in (\ref{P924b}) can be decomposed as 
\begin{align}
\label{P924c}
qL_{X\wedge Y} i_Z\omega=q(-1)^{kl}L_Xi_{Z\wedge Y}\omega-q(-1)^li_{Z\wedge [X,Y]}\omega+qL_Yi_{Z\wedge X}\omega.
\end{align}
From (\ref{CovDerDiffForm}), we have
\begin{align}
\label{P924d}
(-1)^{(k-1)(m-1)}L_Xi_{Z\wedge Y}\omega&=\nabla_X\omega(Z\wedge Y)+\omega(\nabla_X(Z\wedge Y))\\
\label{P924e}
(-1)^{(l-1)(m-1)}L_Y i_{Z\wedge X}\omega&=\nabla_Y\omega(Z\wedge X)+\omega(\nabla_Y(Z\wedge X)). 
\end{align}
Substituting (\ref{P924d}) and (\ref{P924e}) into (\ref{P924c}) and using the fact that 
\begin{equation}
\nonumber
q=(-1)^{(k-1)(m-1)}(-1)^{l(m-1)}=(-1)^{(l-1)(m-1)}(-1)^{k(m-1)}
\end{equation}
gives
\begin{align}
\nonumber
qL_{X\wedge Y} i_Z\omega&=(-1)^{l(m-1)}(-1)^{kl}\left(\nabla_X\omega(Z\wedge Y)+\omega(\nabla_X(Z\wedge Y))\right)\\
\nonumber
&+(-1)^{k(m-1)}\left(\nabla_Y\omega(Z\wedge X)+\omega(\nabla_Y(Z\wedge X))\right)\\
\label{P924f}
&-q(-1)^l\omega(Z\wedge [X,Y]).
\end{align}
Since $\nabla$ is induced, (\ref{P924f}) is further expanded as
\begin{align}
\nonumber
qL_{X\wedge Y} i_Z\omega&=(-1)^{l(m-1)}(-1)^{kl}\nabla_X\omega(Z\wedge Y)\\
\nonumber
&+(-1)^{l(m-1)}(-1)^{kl}\omega((\nabla_X Z)\wedge Y)\\
\nonumber
&+(-1)^{l(m-1)}(-1)^{kl}(-1)^{(k-1)t}\omega(Z\wedge \nabla_X Y)\\
\nonumber
&+(-1)^{k(m-1)}\nabla_Y\omega(Z\wedge X)\\
\nonumber
&+(-1)^{k(m-1)}\omega((\nabla_YZ)\wedge X))\\
\nonumber
&+(-1)^{k(m-1)}(-1)^{(l-1)t}\omega(Z\wedge \nabla_Y X)\\
\label{P924g}
&-q(-1)^l\omega(Z\wedge [X,Y]).
\end{align}
Swapping all the wedge products in (\ref{P924g}) with the appropriate signs gives
\begin{align}
\nonumber
qL_{X\wedge Y} i_Z\omega&=(-1)^{l}\nabla_X\omega(Y\wedge Z)+(-1)^{kl}\omega(Y\wedge \nabla_X Z)\\
\nonumber
&+(-1)^{l}\omega((\nabla_X Y)\wedge Z)+(-1)^{k(l-1)}\nabla_Y\omega(X\wedge Z)\\
\nonumber
&+\omega(X\wedge \nabla_YZ)+(-1)^{k(l-1)}\omega((\nabla_Y X)\wedge Z)\\
\label{P924h}
&-(-1)^l\omega([X,Y]\wedge Z).
\end{align}
(\ref{P924h}) can be rewritten as 
\begin{align}
\nonumber
qL_{X\wedge Y} i_Z\omega&=(-1)^{l}i_Y(\nabla_X\omega)(Z)+(-1)^{kl}\omega(Y\wedge \nabla_X Z)\\
\nonumber
&+(-1)^{k(l-1)}i_X(\nabla_Y\omega)(Z)+\omega(X\wedge \nabla_YZ)\\
\label{P924i}
&+(-1)^l\omega(T(X,Y)\wedge Z),
\end{align}
where $T$ is the higher torsion of $\nabla$.  Since $\nabla$ is torsion-free, the last term vanishes and we obtain
\begin{align}
\nonumber
qL_{X\wedge Y} i_Z\omega&=(-1)^{l}i_Y(\nabla_X\omega)(Z)+(-1)^{kl}\omega(Y\wedge \nabla_X Z)\\
\label{P924j}
&+(-1)^{k(l-1)}i_X(\nabla_Y\omega)(Z)+\omega(X\wedge \nabla_YZ).
\end{align}
Since $\nabla$ is induced, the second term in (\ref{P924b}) decomposes as
\begin{equation}
\label{P924k}
\omega(\nabla_{X\wedge Y} Z)=\omega(X\wedge \nabla_Y Z)+(-1)^{kl}\omega(Y\wedge \nabla_X Z).
\end{equation}
Subsituting (\ref{P924j}) and (\ref{P924k}) into (\ref{P924b}) proves (\ref{P924a}) for the case when $m\ge k+l-1$.  
\end{proof}

\section{Higher Connections and Associative Bilinear Forms}
\noindent Throughout this section, let $n:=\dim M$.
\begin{definition}
\label{DefBM}
Let $\mathcal{B}(M)$ be the set of all fiberwise $\mathbb{R}$-bilinear forms $\eta$ on $\wedge^\bullet TM$ such that 
\begin{itemize}
\item[(i)] $\eta(x\wedge y,z)=\eta(x,y\wedge z)$ for all $x\in \wedge^k T_pM$, $y\in \wedge^l T_pM$, $z\in \wedge^m T_pM$, $p\in M$ with $k,l,m\ge 0$,
\item[(ii)] $\eta$ is smooth in the sense that for all $X\in A^k(M)$, $Y\in A^l(M)$, the function $\eta(X,Y)(p):=\eta(X_p,Y_p)$ for $p\in M$ is smooth.
\end{itemize}
\end{definition}
\begin{corollary}
\label{CorBM1}
Let $\eta\in \mc{B}(M)$.  For all $p\in M$, $\eta$ satisfies
\begin{itemize}
\item[(a)] $\eta(x,y)=(-1)^{kl}\eta(y,x)$ for $x\in \wedge^k T_pM$, $y\in \wedge^l T_pM$
\item[(b)] $\eta(x,y)=0$ for $x\in \wedge^k T_pM$, $y\in \wedge^l T_pM$ with $k+l>n$.
\end{itemize}
\end{corollary}
\begin{proof}
Let $1\in \wedge^0T_pM:=\mathbb{R}$.  For (a), we have
\begin{align}
\nonumber
\eta(x,y)&=\eta(x,y\wedge 1)\\
\nonumber
&=\eta(x\wedge y,1)\\
\nonumber
&=(-1)^{kl}\eta(y\wedge x,1)\\
\nonumber
&=(-1)^{kl}\eta(y,x\wedge 1)\\
\nonumber
&=(-1)^{kl}\eta(y,x).
\end{align}
For $k+l>n$, $x\wedge y=0$ and the second equality in the above calculation implies that $\eta(x,y)=0$.  This proves (b).     
\end{proof}
\begin{proposition}
\label{BMDiffFormProp}
There is a one to one correspondence between $\mathcal{B}(M)$ and  $\Omega^\bullet(M):=\bigoplus_{k=0}^n\Omega^k(M)$.  This correspondence is given by associating $\eta\in \mc{B}(M)$ with $\{\omega^{(k)}\}_{k=0}^n$ where 
\begin{itemize}
\item[(1)] $\omega^{(k)}\in \Omega^k(M)$, for $k=0,\dots, n$, 
\item[(2)] $\eta(x,y)=\omega^{(k+l)}(x\wedge y)$ $\forall ~x\in \wedge^kT_pM$, $y\in \wedge^lT_pM$, $p\in M$, $1\le k+l\le n$, and
\item[(3)] $\eta(x,y)=(xy)\omega_p^{(0)}$ $\forall~x,y\in \wedge^0 T_pM:=\mathbb{R}$, $p\in M$.
\end{itemize}
\end{proposition}
\begin{proof}
Let $\omega\in \Omega^\bullet(M)$ and decompose $\omega$ as 
\begin{equation}
\nonumber
\omega=\omega^{(0)}+\omega^{(1)}+\cdots +\omega^{(n)},
\end{equation}
where $\omega^{(k)}\in \Omega^k(M)$. For $p\in M$, $x\in \wedge^k T_pM$, $y\in \wedge^lT_pM$ with $1\le k+l\le n$, $\eta$ is defined via
\begin{equation}
\nonumber
\eta(x,y):=\omega^{(k+l)}_p(x\wedge y).
\end{equation}
Then for $z\in \wedge^m T_pM$ with $1 \le k+l+m\le n$, we have
\begin{align}
\nonumber
\eta(x\wedge y,z)&=\omega^{(k+l+m)}_p((x\wedge y)\wedge z)\\
\nonumber
&=\omega^{(k+l+m)}_p(x\wedge (y\wedge z))\\
\nonumber
&=\eta(x,y\wedge z).
\end{align} 
For $p\in M$, $x,y\in \wedge^0T_pM:=\mathbb{R}$, $\eta$ is defined by $\eta(x,y):=(xy)\omega^{(0)}_p$.  For $z\in \wedge^0T_pM:=\mathbb{R}$, we have 
\begin{equation}
\nonumber
\eta(xy,z):=((xy)z)\omega^{(0)}_p=(x(yz))\omega^{(0)}_p=\eta(x,yz).
\end{equation}
For $x\in \wedge^kT_pM$, $y\in \wedge^lT_pM$, $p\in M$ with $k+l>n$, we set $\eta(x,y):=0$.  This proves (i) of Definition \ref{DefBM}.  (ii) of Definition \ref{DefBM} follows from the fact that $\omega$ is smooth.  

Now let $\eta\in \mathcal{B}(M)$.  For $p\in M$, $x\in \wedge^kT_pM$ with $k\ge 1$, we define
\begin{equation}
\nonumber
\omega^{(k)}_p(x):=\eta(x,1),
\end{equation}
where we identify $\omega_p^{(k)}$ with an element of $\wedge^k T^\ast_pM$ via the natural isomorphism $(\wedge^k T_pM)^\ast\simeq \wedge^k T^\ast_pM$.  For $k=0$, we define $\omega_p^{(0)}:=\eta(\mathbf{1},\mathbf{1})(p)$ where $\mathbf{1}\in A^0(M)=C^\infty(M)$ is the constant function $p\mapsto 1\in \mathbb{R}$.  (ii) of Definition \ref{DefBM} implies that $\omega^{(k)}$ is smooth.  

Let $\varphi: \Omega^\bullet(M)\rightarrow \mathcal{B}(M)$ and $\psi: \mathcal{B}(M)\rightarrow \Omega^\bullet(M)$ be the two mappings constructed above.  Its clear then that $\psi\circ \varphi=id_{\Omega^\bullet(M)}$ and $\varphi\circ \psi=id_{\mathcal{B}(M)}$.  This proves the proposition.   
\end{proof}
\noindent We now give a necessary and sufficient condition for $\eta\in \mc{B}(M)$ to be nondegenerate on the fibers of $\wedge^\bullet TM$, that is, for any $p\in M$ and any $v\in \wedge^\bullet T_pM$, we have
\begin{equation}
\nonumber
\eta(u,v)=0~\forall~u\in \wedge^\bullet T_pM \Longleftrightarrow ~v=0.
\end{equation}

\begin{proposition}
\label{NondegenBM}
Let $\eta\in \mc{B}(M)$ and let $\{\omega^{(k)}\}_{k=0}^n$ be the set of differential forms associated with $\eta$ by Propositon \ref{BMDiffFormProp}.  Then $\eta$ is nondegenerate on the fibers of $\wedge^\bullet TM$ iff $\omega^{(n)}$ is a volume form.  
\end{proposition}
\begin{proof}
$(\Rightarrow)$ Suppose $\eta$ is nondegenerate on the fibers of $\wedge^\bullet TM$.  Let $p$ be any element of $M$ and let $v$ be any nonzero element of $\wedge^n T_pM$.  Since $\eta$ is nondegenerate, we have 
\begin{equation}
\nonumber
\omega^{(n)}_p(v)=\omega^{(n)}_p(v\wedge 1)=\eta(v,1)\neq 0.
\end{equation}
Since $p\in M$ was arbitrary, this proves that $\omega^{(n)}$ is a volume form.

($\Leftarrow$) Suppose $\omega^{(n)}$ is a volume form.  Let $p$ be any element of $M$ and let $v$ be any nonzero element in $\wedge^\bullet T_pM$.  Then $v$ can be decomposed as
\begin{equation}
\nonumber
v=v^{(0)}+v^{(1)}+\cdots + v^{(n)},
\end{equation}
where $v^{(i)}\in \wedge^iT_pM$ for $i=0,\dots, n$.  Since $v\neq 0$, there exists some $i$ such that $v^{(i)}\neq 0$. Let $k:=\min\{i~|~v^{(i)}\neq 0\}$.  By Propositon \ref{Multilinear3}, there exists an $(n-k)$-vector $u\in \wedge^{(n-k)}T_pM$ such that $u\wedge v^{(k)}\in \wedge^nT_pM$ is nonzero.  For $i>k$, we clearly have $u\wedge v^{(i)}  =0$.  This gives
\begin{align}
\nonumber
\eta(u,v)&=\eta(u\wedge v,1)\\
\nonumber
&=\omega^{(n)}(u\wedge v^{(k)})\\
\nonumber
&\neq 0,
\end{align}
where the last line follows from the fact that $\wedge^nT_pM$ is a 1-dimensional space generated by $u\wedge v^{(k)}\neq 0$ and $\omega^{(n)}$ is non-vanishing.  This completes the proof. 
\end{proof}

\begin{corollary}
\label{FrobeniusAlgCor}
Let $\eta\in \mc{B}(M)$ and let $\{\omega^{(k)}\}$ be the set of differential forms associated with $\eta$ by Proposition \ref{BMDiffFormProp}.  Then $(\wedge^\bullet T_pM,\eta_p)$ is a supercommutative Frobenius algebra for all $p\in M$ iff $\omega^{(n)}$ is a volume form.
\end{corollary}
\begin{proof}
Recall that a Frobenius algebra is a finite dimensional unital associative algebra $A$ with a nondegenerate bilinear form $\langle\cdot,\cdot\rangle$ satisfying
\begin{equation}
\label{FrobeniusAssociative}
\langle ab,c\rangle =\langle a, bc\rangle\hspace*{0.1in} \forall a,b,c\in A.
\end{equation}
Note that $\wedge^\bullet T_pM$ is naturally a finite dimensional unital associative algebra under the exterior product.  In addition, the bilinear form $\eta_p$ satisfies the associativity condition of  (\ref{FrobeniusAssociative}).  For $(\wedge^\bullet T_pM,\eta_p)$ to be a Frobenius algbera, we only need $\eta_p$ to be nondgenerate.  By Proposition \ref{NondegenBM}, this happens precisely when $\omega^{(n)}$ is a volume form.  Lastly, note that $\wedge^\bullet T_pM$ is naturally a $\mathbb{Z}_2$-graded algebra via
\begin{equation}
\nonumber
\wedge^\bullet T_pM=\left(\wedge^\bullet T_pM\right)_0\oplus (\wedge^\bullet T_pM)_1
\end{equation}
where 
\begin{equation}
\nonumber
\left(\wedge^\bullet T_pM\right)_0:=\bigoplus_{k=0}\wedge^{2k} T_pM,\hspace*{0.2in}\left(\wedge^\bullet T_pM\right)_1:=\bigoplus_{k=0}\wedge^{2k+1} T_pM.
\end{equation}
With the above $\mathbb{Z}_2$-grading, the relation $x\wedge y=(-1)^{kl}y\wedge x$ for $x\in \wedge^kT_pM$ and $y\in \wedge^l T_pM$ is precisely the conditon of supercommutativity.  
\end{proof}

\begin{proposition}
\label{closedBM}
Let $\eta\in \mc{B}(M)$ and let $\{\omega^{(k)}\}_{k=0}^n$ be the set of differential forms associated with $\eta$ by Proposition \ref{BMDiffFormProp}. If $d\omega^{(k)}=0$ for all $k$, then for any decomposable $(k+1)$-vector field $X=X_1\wedge \cdots\wedge X_{k+1}$, $\eta$ satisfies
\begin{equation}
\label{closedBMFormula}
\sum_{i=1}^{k+1} (-1)^i X_i(\eta(X[i],\textbf{1}))=\sum_{1\le i<j\le k+1} (-1)^{i+j}\eta([X_i,X_j],X[i,j])
\end{equation}
where $\textbf{1}$ denotes the function on $M$ whose value is always $1\in\mathbb{R}$ and 
\begin{align}
\nonumber
X[i]&:=X_1\wedge \cdots \wedge \widehat{X_i}\wedge \cdots \wedge X_{k+1}\\
\nonumber
X[i,j]&:=X_1\wedge \cdots \wedge \widehat{X_i}\wedge \cdots\wedge\widehat{X_j}\wedge \cdots  \wedge X_{k+1}.
\end{align}
In addition, if $M$ is connected, then $\eta(\textbf{1},\textbf{1})$ is constant on $M$.  
\end{proposition}
\begin{proof}
This follows from the invariant formula for the exterior derivative \cite{Lee}.  Specifically,
\begin{align}
\nonumber
d\omega^{(k)}(X)&=\sum_{i=1}^{k+1} (-1)^{i-1}X_i(\omega^{(k)}(X[i]))+\sum_{1\le i<j\le k+1}(-1)^{i+j}\omega^{(k)}([X_i,X_j]\wedge X[i,j])\\
\nonumber
&=\sum_{i=1}^{k+1} (-1)^{i-1}X_i(\eta(X[i],\textbf{1}))+\sum_{1\le i<j\le k+1}(-1)^{i+j}\eta([X_i,X_j], X[i,j]).
\end{align}
The first part then follows from the fact that $d\omega^{(k)}=0$.  For the second part, suppose that $M$ is connected.  Then for any $Y\in A^1(M)$, we have
\begin{align}
d\omega^{(0)}(Y)=Y\omega^{(0)}=Y\eta(\textbf{1},\textbf{1}).
\end{align}
Since $d\omega^{(0)}=0$, it follows that $\eta(\textbf{1},\textbf{1})$ is locally constant on $M$.   Since $M$ is connected, $\eta(\textbf{1},\textbf{1})$  must be constant on all of $M$.
\end{proof}

Using Proposition \ref{BMDiffFormProp}, there is a natural way to define the covariant derivative of $\eta$ with respect to a higher connection $\nabla$.  For $X\in A^k(M)$, $Y\in A^l(M)$, and $Z\in A^m(M)$ (with $k+l+m\le n+1$), we define 
\begin{align}
\label{CovDerBM}
(\nabla_X\eta)(Y,Z):=(\nabla_X\omega^{(t)})(Y\wedge Z).
\end{align}
where $t=k+l+m-1$ and $\{\omega^{(i)}\}$ is the set of differential forms associated with $\eta$ by Proposition \ref{BMDiffFormProp}.  
\begin{definition}
\label{ParallelBMDef}
Let $\eta\in \mc{B}(M)$ and let $\nabla$ be a higher connection on $M$.  $\eta$ is \textit{parallel} with respect to $\nabla$ if $(\nabla_X\eta)(Y,Z)=0$ for all $X\in A^k(M)$, $Y\in A^l(M)$, and $Z\in A^m(M)$, where $k>0$ and $l+m>0$.  This condition is denoted symbolically by $\nabla\eta\equiv 0$.
\end{definition}
\begin{remark}
Note that when $l+m=0$ in Definition \ref{ParallelBMDef}, we have $Y,Z\in C^\infty(M)$.  Let $f:=Y$ and $g:=Z$. Then for $X\in A^k(M)$ $(k>0)$, Proposition \ref{P925} implies that
\begin{align}
\nonumber
(\nabla_X\eta)(f,g)&=(\nabla_X\omega^{(k-1)})(fg)\\
\nonumber
&=fg(\nabla_X\omega^{(k-1)})\\
\nonumber
&=fgL_X\omega^{(k-1)}.
\end{align}
In other words, when $l+m=0$, $(\nabla_X\eta)(f,g)$ is completely determined by the Lie derivative of $\omega^{(k-1)}$ along $X$ with no contribution from the higher connection.
\end{remark}
\noindent The next result gives a necessary and sufficient condition for $\nabla \eta\equiv 0$:  
\begin{proposition}
\label{ParallelBMProp}
Let $\eta\in \mc{B}(M)$ and let $\nabla$ be a higher connection on $M$.  In addition, let $\{\omega^{(i)}\}$ be the set of differential forms associated with $\eta$ by Proposition \ref{BMDiffFormProp} and let $(\widetilde{\nabla},\{F^{k,l}\})$ be the unique pair associated with $\nabla$ by Theorem \ref{HigherConnectionData}.  Then $\nabla\eta\equiv 0$ (in the sense of Definition \ref{ParallelBMDef}) iff
\begin{equation}
\label{ParallelBMProp1}
(\widetilde{\nabla}_X\omega^{(t)})(Y)=\omega^{(t)}(F^{k,l}(X,Y))
\end{equation}
for all $X\in A^k(M)$, $Y\in A^l(M)$ with $k,l>0$, where $t=k+l-1\le n:=\dim M$.
\end{proposition}
\begin{proof}
From the definition of $(\nabla_X\eta)(Y,Z)$ it suffices to consider the case where $Y\in A^l(M)$ with $l>0$ and $Z=\textbf{1}\in \Omega^0(M):=C^\infty(M)$, where $\textbf{1}$ is the function whose value is always $1\in \mathbb{R}$.  Then
\begin{align}
\nonumber
(\nabla_X\eta)(Y,\textbf{1})&=(\nabla_X\omega^{(t)})(Y\wedge \textbf{1})\\
\nonumber
&=(\nabla_X\omega^{(t)})(Y)\\
\nonumber
&=(-1)^{(k-1)(t-1)}L_Xi_Y\omega^{(t)}-\omega^{(t)}(\nabla_XY)\\
\nonumber
&=(-1)^{(k-1)(t-1)}L_Xi_Y\omega^{(t)}-\omega^{(t)}(\widetilde{\nabla}_XY)-\omega^{(t)}(F^{k,l}(X,Y))\\
\nonumber
&=(\widetilde{\nabla}_X\omega^{(t)})(Y)-\omega^{(t)}(F^{k,l}(X,Y)),
\end{align}
where the last equality follows from (\ref{CovDerDiffForm}).  From this, we see that $\nabla\eta\equiv 0$ iff (\ref{ParallelBMProp1}) is satisfied.
\end{proof}
\begin{remark}
Since $F^{1,1}\equiv 0$, we see from Proposition \ref{ParallelBMProp} that $\nabla\eta\equiv 0$ implies that $\omega^{(1)}$ is parallel with respect to the affine connection $\widetilde{\nabla}$, that is, $(\widetilde{\nabla}_X\omega^{(1)})(Y)= 0$  for all $X,Y\in A^1(M)$.
\end{remark}
\noindent Proposition \ref{ParallelBMProp} suggests that an induced higher connection has little chance of satisfying $\nabla\eta\equiv 0$.  Without placing very restrictive conditions on $\eta$, any higher connection satisfying $\nabla\eta\equiv 0$ will, in general, have non-zero twist fields.  In other words, a non-induced higher connection is needed to satisfy $\nabla\eta\equiv 0$.  So we see that by equipping the full exterior bundle $\wedge^\bullet TM$ with an associative bilinear form $\eta$ and then demanding that $\nabla\eta\equiv 0$, the notion of a non-induced higher connection is required.  

Given $\eta\in \mc{B}(M)$, we seek a higher connection $\nabla$ such that
\begin{itemize}
\item[(a)] $\nabla\eta\equiv 0$
\item[(b)] $T^\nabla\equiv 0$ (where $T^\nabla$ denotes the higher torsion of $\nabla$).
\end{itemize}
Let $(\widetilde{\nabla},\{F^{k,l}\})$ be the unique pair associated with $\nabla$ by Theorem \ref{HigherConnectionData}.  By Theorem \ref{TorsionThmB}, $\nabla$ is torsion-free iff $\widetilde{\nabla}$ is torsion-free as an affine connection on $TM$, and the twist fields satisfy 
\begin{equation}
\label{TwistFieldsTorsionRecall}
F^{k,l}(X,Y)=(-1)^{(k-1)(l-1)}F^{l,k}(Y,X)
\end{equation}
for $X\in A^k(M)$, $Y\in A^l(M)$, with $k,l> 0$ and $k+l-1\le n:=\dim M$.  Let $\{\omega^{(t)}\}$ be the unique set of differential forms associated with $\eta$ by Proposition \ref{BMDiffFormProp}.  It follows from Proposition \ref{ParallelBMProp} and equation (\ref{TwistFieldsTorsionRecall}) that any higher connection which satisfies (a) and (b) simultaneously must also satisfy the condition
\begin{equation}
\label{NecessaryConditionAB}
(\widetilde{\nabla}_X\omega^{(t)})(Y)=(-1)^{(k-1)(l-1)}(\widetilde{\nabla}_Y\omega^{(t)})(X)
\end{equation}
for all $X\in A^k(M)$, $Y\in A^l(M)$, with $k,l> 0$ and $k+l-1\le n$, where $t:=k+l-1$.  Unfortunately, equation (\ref{NecessaryConditionAB}) does not hold in general, which means having a higher connection which satisfies (a) and (b) simultaneously is not possible.  However, all is not quite lost, since (\ref{NecessaryConditionAB}) does hold when $X\wedge Y=0$ as the next result shows. 

\begin{proposition}
\label{FlipXYProp}
Let $\widetilde{\nabla}$ be a torsion free affine connection on $TM$ and extend $\widetilde{\nabla}$ to an induced higher connection.  In addition, let $X\in A^k(M)$, $Y\in A^l(M)$, and $\omega\in \Omega^{k+l-1}(M)$, with $k,l>0$ and $k+l-1\le n$.  If $X\wedge Y=0$, then 
\begin{equation}
\nonumber
(\widetilde{\nabla}_X\omega)(Y)=(-1)^{(k-1)(l-1)}(\widetilde{\nabla}_Y\omega)(X).
\end{equation}
\end{proposition}

\begin{proof}
Let $t:=k+l-1$.  Using  (\ref{CovDerDiffForm}) and the fact that $(-1)^{(k-1)(t-1)}=(-1)^{(k-1)l}$ and $(-1)^{(l-1)(t-1)}=(-1)^{(l-1)k}$ gives
\begin{align}
\label{FlipXY1}
(\widetilde{\nabla}_X\omega)(Y)&=(-1)^{(k-1)l}L_Xi_Y\omega-\omega(\widetilde{\nabla}_XY)
\end{align}
and 
\begin{align}
\label{FlipXY2}
(\widetilde{\nabla}_Y\omega)(X)&=(-1)^{(l-1)k}L_Yi_X\omega-\omega(\widetilde{\nabla}_YX).
\end{align}
Applying Proposition \ref{LieDerivativeProp}-(ii) to (\ref{FlipXY1}) gives 
\begin{align}
\label{FlipXY3}
(\widetilde{\nabla}_X\omega)(Y)&=\omega([X,Y])+i_YL_X\omega-\omega(\widetilde{\nabla}_XY).
\end{align}
Next, multiply (\ref{FlipXY2}) by $s:=(-1)^{(k-1)(l-1)}$ to obtain
\begin{align}
\label{FlipXY4}
s(\widetilde{\nabla}_Y\omega)(X)&=(-1)^{l-1}L_Yi_X\omega-s\omega(\widetilde{\nabla}_YX).
\end{align}
Since $X\wedge Y=0$ (by hypothesis), Proposition \ref{LieDerivativeProp}-(iv) implies
\begin{align}
\label{FlipXY5}
(-1)^{l-1}L_Yi_X\omega=i_YL_X\omega.
\end{align} 
Subsituting (\ref{FlipXY5}) into (\ref{FlipXY4}) gives 
\begin{equation}
\label{FlipXY6}
s(\widetilde{\nabla}_Y\omega)(X)=i_YL_X\omega -s\omega(\widetilde{\nabla}_YX).
\end{equation}
Since $\widetilde{\nabla}$ is torsion-free as an affine connection, its higher torsion (as an induced higher connection) also vanishes by Proposition \ref{VanishingTorsionProp}.  Hence, (\ref{FlipXY3}) can be rewritten as
\begin{equation}
\label{FlipXY7}
(\widetilde{\nabla}_X\omega)(Y)=i_YL_X\omega-s\omega(\widetilde{\nabla}_YX).
\end{equation}
A quick comparison of (\ref{FlipXY7}) and (\ref{FlipXY6}) completes the proof.  
\end{proof}
\noindent Motivated by Proposition \ref{FlipXYProp} and the above discussion, we introduce the notion of \textit{almost torsion-free}:
\begin{definition}
\label{AlmostTorsionFreeDef}
Let $\nabla$ be a higher connection and let $T$ denote its higher torsion.  Then $\nabla$ is \textit{almost torsion-free} if
\begin{itemize}
\item[1.] $T(X,Y)=0$ for all $X,Y\in A^1(M)$; and
\item[2.] $T(X,Y)=0$ for all $X\in A^k(M)$, $Y\in A^l(M)$ such that $X\wedge Y=0$. 
\end{itemize}
\end{definition}
\noindent To set up our next result, let 
\begin{equation}
\nonumber
\varphi^{(k)}: \mc{B}(M)\rightarrow \Omega^k(M),
\end{equation}
be the map given by $\eta\mapsto \omega^{(k)}$, where $\{\omega^{(t)}\}_{t=0}^n$ is the set of differential forms associated with $\eta$ by Proposition \ref{BMDiffFormProp}.  Let $\mc{B}^\circ(M)\subset \mc{B}(M)$ be the set of all $\eta\in \mc{B}(M)$ such that
\begin{itemize}
\item[1.] $\varphi^{(1)}(\eta)\equiv 0$
\item[2.] for $t>1$, if $\omega^{(t)}:=\varphi^{(t)}(\eta)\neq 0$, then $\omega^{(t)}$ is also nonvanishing, that is, $\omega^{(t)}_p\neq 0$ for all $p\in M$. 
\end{itemize}
\begin{theorem}
\label{HigherConnBMThm}
For any $\eta\in \mc{B}^\circ(M)$, $M$ admits an almost torsion-free higher connection $\nabla$ such that $\nabla\eta\equiv 0$ in the sense of Definition \ref{ParallelBMDef}.
\end{theorem}
\begin{proof}
Let $\omega^{(t)}:=\varphi^{(t)}(\eta)$ for $t=0,1,\dots, n$ and let $\mc{T}:=\{t~|~\omega^{(t)}\neq 0\}$.  Fix a Riemannian metric $g$ on $M$.  Let $\overline{g}$ denote the inverse metric on $T^\ast M$.  Recall that if $g$ is expressed in local coordinates as 
\begin{equation}
\nonumber
g=\sum_{i,j} g_{ij}dx^i\otimes dx^j,
\end{equation}
then 
\begin{equation}
\nonumber
\overline{g}=\sum_{i,j}g^{ij}\frac{\partial}{\partial x^i}\otimes \frac{\partial}{\partial x^j}
\end{equation}
where $(g^{ij})=(g_{ij})^{-1}$.  For decomposable $k$-forms $\omega:=\omega^1\wedge \cdots \wedge \omega^k$ and $\phi:=\phi^1\wedge \cdots \wedge \phi^k$ define
\begin{equation}
\nonumber
\langle\omega,\phi\rangle := \det(\overline{g}(\phi^i,\omega^j)).
\end{equation}
By linearity, $\langle\cdot, \cdot\rangle$ extends to a smooth, symmetric, positive definite bilinear form on the exterior bundle $\wedge^k T^\ast M$.  For each $t\in \mc{T}$ with $t>0$, let $E^{(t)}\in A^t(M)$, be the $t$-vector field defined by
\begin{equation}
\label{EtDef}
E^{(t)}:=\frac{1}{\langle\omega^{(t)},\omega^{(t)}\rangle}(\omega^{(t)})^\sharp,
\end{equation}
where $(\omega^{(t)})^\sharp$ is the $t$-vector field obtained by raising the indices of $\omega$ with $\overline{g}$.  Recall that by hypothesis, $\omega^{(t)}$ is non-vanishing on $M$.   Hence, $\langle\omega^{(t)},\omega^{(t)}\rangle|_{p}\neq 0$ for all $p\in M$.  (\ref{EtDef}) then implies
\begin{align}
\label{EtIdentity}
\omega^{(t)}(E^{(t)})&=\frac{1}{\langle\omega^{(t)},\omega^{(t)}\rangle}\omega^{(t)}((\omega^{(t)})^\sharp)=\frac{1}{\langle\omega^{(t)},\omega^{(t)}\rangle}\langle\omega^{(t)},\omega^{(t)}\rangle=1.
\end{align}
We will now use $\{E^{(t)}\}$ to construct an almost torsion-free higher connection satisfying $\nabla\eta\equiv 0$.  

By Theorem \ref{HigherConnectionData}, a higher connection is determined by an affine connection $\widetilde{\nabla}$ on $TM$ and a set of twist fields $\{F^{k,l}\}$ where $k,l>0$ and $k+l-1\le n:=\dim M$.  To obtain the desired higher connection, let $\widetilde{\nabla}$ be any torsion-free affine connection on $TM$ (e.g., take $\widetilde{\nabla}$ to be the Levi-Civita connection associated to $g$).  By Proposition \ref{ParallelBMProp}, the twist fields must be chosen so that
\begin{equation}
\label{ChooseFklA}
(\widetilde{\nabla}_X\omega^{(t)})(Y)=\omega^{(t)}(F^{k,l}(X,Y))
\end{equation}
for all $X\in A^k(M)$, $Y\in A^l(M)$ with $k,l>0$, where $t=k+l-1\le n:=\dim M$. We will now construct a complete set of twist fields $\{F^{k,l}\}$ ($k,l>0$, $k+l-1\le n$) which satisfy (\ref{ChooseFklA}).

Now, for $t\notin \mc{T}$ ($t\le n$), we have $\omega^{(t)}\equiv 0$ by definition.  Consequently, for each $t\notin \mc{T}$ we can set $F^{k,l}\equiv 0$ for all $k,l>0$ satisfying $k+l-1=t$.  This clearly satisfies (\ref{ChooseFklA}).  Note that by hypothesis, $1\notin \mc{T}$, and with our choice, we also have $F^{1,1}\equiv 0$ (as required by Theorem \ref{HigherConnectionData}).  

Now for each $t\in \mc{T}$, we define $F^{k,l}$ with $k,l>0$ and $k+l-1=t$ by  
\begin{equation}
\label{ChooseFklB}
F^{k,l}(X,Y):=\left[(\widetilde{\nabla}_X\omega^{(t)})(Y)\right]E^{(t)}\in A^{k+l-1}(M)
\end{equation}
for all $X\in A^k(M)$, $Y\in A^l(M)$.  Note that $(\widetilde{\nabla}_X\omega^{(t)})(Y)\in C^\infty(M)$, and that $F^{k,l}(fX,Y)=F^{k,l}(X,fY)=fF^{k,l}(X,Y)$ for all $f\in C^\infty(M)$ by Theorems \ref{CovDerDiffFormProp} and \ref{CovDerDiffFormProp2}.  It follows easily from (\ref{EtIdentity}) that (\ref{ChooseFklB}) satisfies  (\ref{ChooseFklA}).  This completes the construction of the twist fields.  

With $\widetilde{\nabla}$ and $\{F^{k,l}\}$ in hand, we define $\nabla$ to be the higher connection which is uniquely determined by the pair  $(\widetilde{\nabla},\{F^{k,l}\})$. Proposition \ref{ParallelBMProp} then implies that $\nabla\eta\equiv 0$ in the sense of Definition \ref{ParallelBMDef}.  All that remains to be done now is to show that $\nabla$ is almost torsion-free.  To do this, let $T$ denote the higher torsion of $\nabla$. Since $\widetilde{\nabla}$ is torsion-free, we have $T(X,Y)=0$ for all $X, Y\in A^1(M)$, which is the first condition of Definition \ref{AlmostTorsionFreeDef}.  To verify the second condition, let $X\in A^k(M)$, $Y\in A^l(M)$ with $k,l>0$, $k+l-1\le n$, and $X\wedge Y=0$.  With $\widetilde{\nabla}$ torsion-free, the proof of Theorem \ref{TorsionThmB} shows that $T(X,Y)=0$ iff the twist fields satisfy
\begin{equation}
\label{CheckAlmostTorsionFree1}     
F^{k,l}(X,Y)=(-1)^{(k-1)(l-1)}F^{l,k}(Y,X).
\end{equation}
We now verify (\ref{CheckAlmostTorsionFree1}).  For $k+l-1\notin \mc{T}$, we have
\begin{equation}
F^{k,l}(X,Y)=0=(-1)^{(k-1)(l-1)}F^{l,k}(X,Y).
\end{equation}
For $t=k+l-1\in \mc{T}$, we have
\begin{align}
\nonumber
F^{k,l}(X,Y)&=[(\widetilde{\nabla}_X\omega^{(t)})(Y)]E^{(t)}\\
\nonumber
&=(-1)^{(k-1)(l-1)}[(\widetilde{\nabla}_Y\omega^{(t)})(X)]E^{(t)}\\
\nonumber
&=(-1)^{(k-1)(l-1)}F^{l,k}(Y,X),
\end{align}
where the second equality follows from Proposition \ref{FlipXYProp}.  This completes the proof.
\end{proof}
Using Theorem \ref{HigherConnBMThm}, we can relate higher connections to \textit{multisymplectic geometry} \cite{CIL} (or \textit{higher symplectic geometry} if one follows the terminology of \cite{Rog}).   To start, we recall the notion of a multisymplectic form of degree $t+1$ (or, more concisely, a $t$-plectic form):
\begin{definition}
A $(t+1)$-form $\omega$ on $M$ is multisymplectic of degree $t+1$ (or $t$-plectic) if it satisfies the following two conditions:
\begin{itemize}
\item[(i)] $d\omega=0$
\item[(ii)] $\omega$ is non-degenerate in the sense that for all $p\in M$ and $v\in T_pM$, 
\begin{equation}
\nonumber
i_v\omega=0 \Leftrightarrow v=0.
\end{equation}
The pair $(M,\omega)$ is then called a multisymplectic manifold of order $t+1$ (or \textit{$t$-plectic manifold}).
\end{itemize}
\end{definition}
\begin{remark}
Following the terminology of \cite{Rog}, a 1-plectic form is just a symplectic form on $M$.
\end{remark}
To relate $t$-plectic forms to higher connections, let $\mc{B}^{plc}(M)$ be the set of all $\eta\in \mc{B}(M)$ such that 
\begin{itemize}
\item[(i)] $\varphi^{(1)}(\eta)\equiv 0$
\item[(ii)]for $t>1$, if $\varphi^{(t)}(\eta)\neq 0$, then $\varphi^{(t)}$ is a $(t-1)$-plectic form.
\end{itemize}
Since a $(t-1)$-plectic form is necessarily non-vanishing, we immediately have $\mc{B}^{plc}(M)\subset \mc{B}^\circ(M)$.  Theorem \ref{HigherConnBMThm} then implies the following:
\begin{corollary}
\label{tPlecticCor}
Let $\eta\in \mc{B}^{plc}(M)$. Then $M$ admits an almost torsion-free higher connection $\nabla$ such that $\nabla\eta\equiv 0$ in the sense of Definition \ref{ParallelBMDef}
\end{corollary}

\section{Conclusion}
In this paper, the notion of higher connections has been introduced as part of a program in differential geometry to extend the familiar constructions and operations for vector fields to multivector fields (MVFs).  The aforementioned program is motivated by generalized geometry and string theory, and is based on the idea of treating the full exterior bundle $\wedge^\bullet TM$ as an extended tangent bundle with the Schouten-Nijenhuis bracket playing the role of the Lie bracket of vector fields.     Consequently, in the context of this program, a higher connection on the full exterior bundle $\wedge^\bullet TM$ is the analogue of an affine connection on the tangent bundle $TM$.  

In section 5, we equipped the full exterior bundle $\wedge^\bullet TM$ with an associative bilinear form $\eta$ and showed that such a structure can be naturally identified with a collection of differential forms $\{\omega^{(t)}\}$ of various degrees.  This fact allowed one to take the covariant derivative of $\eta$ with respect to a higher connection.  The natural problem of finding a higher connection $\nabla$ which satisfies $\nabla\eta\equiv 0$ naturally leads to the notion of a non-induced higher connection; the differential forms associated with $\eta$ determine the twist fields of the higher connection.  For any $\eta\in \mc{B}^\circ(M)$, 
Theorem \ref{HigherConnBMThm} shows that $M$ admits an almost torsion-free higher connection $\nabla$ which satisfies $\nabla\eta\equiv 0$.  However, the higher connection constructed in the proof of Theorem \ref{HigherConnBMThm} is by no means unique or canonical, and this raises the following question:
\begin{itemize}
\item[ ] \textit{What conditions could be placed on the associative bilinear form $\eta$ which would give rise to a unique or canonical higher connection?  In other words, is there a ``best" choice of higher connection?}
\end{itemize}  
Corollary \ref{tPlecticCor}, an immediate consequence of Theorem \ref{HigherConnBMThm}, links higher connections to  multisymplectic geometry by restricting attention to all $\eta$ which are built up from multisymplectic forms of various degrees.  The question raised above as well as the  relationship between higher connections and multisymplectic geometry (which was only touched upon in this paper) will be explored in greater depth as part of future work.

\end{document}